\documentclass{amsart}
\usepackage{mathrsfs}
\usepackage[colorlinks,linkcolor=blue,anchorcolor=blue,citecolor=blue,CJKbookmarks=True]{hyperref}
\usepackage{lineno,hyperref}
\usepackage{color}

\usepackage{tikz}
\usetikzlibrary{positioning, shapes.geometric}
\usepackage{amsmath,amssymb,amsfonts}%
\usepackage{amsthm}%
\usepackage{mathrsfs}%
\usepackage[title]{appendix}%
\usepackage{xcolor}%
\usepackage{graphicx}
\usepackage{subfigure}
\usepackage{float}
\usepackage{diagbox}
\allowdisplaybreaks \allowdisplaybreaks[4]
\newtheorem{Def}{Definition}[section]
\newtheorem{lem}[Def]{Lemma}
\newtheorem{theo}[Def]{Theorem}

\newtheorem{rem}[Def]{Remark}

\newtheorem{assum}{Assumption}
\newtheorem{cor}[Def]{Corollary}
\newtheorem{con}{Condition}
\usepackage{pgf}
\definecolor{Green}{RGB}{0,128,0}

\newcommand{\E}{\mathbb{E}}
\newcommand{\N}{\mathbb{N}}
\newcommand{\LL}{\langle}
\newcommand{\RR}{\rangle}

\newcommand{\R}{\mathbb R}

\newcommand{\mcal}{\mathcal}

\newcommand{\mscr}{\mathscr}
\newcommand{\mbb}{\mathbb}
\newcommand{\mbf}{\mathbf}

\newcommand{\ud}{\mathrm d}
\newcommand{\PD}{\partial}
\numberwithin{equation}{section}

\begin{document}

\title[DSAV method for Langevin dynamics]{A uniform-in-time weakly convergent explicit numerical method for the underdamped Langevin equation with polynomial potentials}


\author{Diancong Jin}
\address{School of Mathematics and Statistics, Huazhong University of Science and Technology;	Hubei Key Laboratory of Engineering Modeling and Scientific Computing, Huazhong University of Science and Technology}
\curraddr{Wuhan 430074, China}
\email{jindc@hust.edu.cn}
\thanks{This work was funded by National Natural Science Foundation of China (No. 12201228).
}



\date{}

\dedicatory{}
\begin{abstract}
The underdamped Langevin equation is a fundamental model in statistical mechanics for sampling Gibbs measures and simulating molecular dynamics, for which numerical methods with uniform-in-time weak convergence are essential for accurately reproducing long-time statistical observables and invariant measures of the underlying dynamics. Currently, such uniform-in-time weak convergence is established for implicit schemes, but remains unknown for explicit ones under polynomially growing potentials. To improve efficiency in long-time simulations, we propose the first explicit numerical method for the underdamped Langevin equation with polynomially growing potentials that is proven to achieve uniform-in-time weak convergence. The explicit numerical method is constructed by introducing a dissipativity on the scalar auxiliary variable (SAV), which we call the DSAV method. The proposed DSAV method enables the approximation of the invariant measure for the underdamped Langevin equation with a precision of $\varepsilon$ at a significantly reduced computational cost of $\mathcal{O}(\varepsilon^{-1} \log(\varepsilon^{-1}))$.
In addition, we establish the existence and positivity of the density function of the numerical solution without using the Malliavin calculus. Numerical experiments are performed to verify the theoretical findings and demonstrate the long-time stability of the proposed numerical method.
\end{abstract}

\maketitle

\section{Introduction} 
\subsection{Research background} \label{Sec:1.1} 
The underdamped Langevin equation describes the motion of a particle subject to friction and a stochastic forcing:
\begin{equation}\label{eq:Langevinorg}
	\left\{
	\begin{aligned}
		\ud \bar{\mathfrak{P}}(t) &= -\nabla V(\bar{\mathfrak{Q}}(t))\,\ud t - v \bar{\mathfrak{P}}(t)\,\ud t +\ud W(t),\\
		\ud \bar{\mathfrak{Q}}(t) &= \bar{\mathfrak{P}}(t)\,\ud t,
	\end{aligned}
	\right.
\end{equation}
where \(V\in C^\infty_{\rm p}(\mathbb{R}^d;\mathbb{R})\) is a confining potential, \(v>0\) is the friction coefficient, and \(W\) is a \(d\)-dimensional Brownian motion according to the fluctuation–dissipation relation \cite{Langevin2,Pav14}. Here, \(\bar{\mathfrak{Q}}(t)\) and \(\bar{\mathfrak{P}}(t)\) denote the particle’s position and momentum, respectively.
The equation \eqref{eq:Langevinorg} arises in a broad range of applications, including chemical kinetics \cite{vanLNM}, stochastic thermodynamics \cite{LeePRE}, nonequilibrium statistical mechanics \cite{ChetriteCMP}, and biomedical engineering \cite{RegevPhysA}. In addition, endowed with rich structural properties and closely connected to Hamiltonian systems, hypoelliptic operators, and ergodic theory, it thereby serves as a fundamental model in stochastic dynamics.

An intrinsic dynamical property of \eqref{eq:Langevinorg} is that it admits a unique invariant measure 
$\ud\widetilde{\mu}=\widetilde{\rho}(p,q)\ud p\ud q$ with $\widetilde{\rho}(p,q)=\frac{1}{Z_0}\exp(-2v(\frac{1}{2}\|p\|^2+V(q)))$, $(p,q)\in\mbb R^{2d}$ with $Z_0$ being the normalizing constant (cf.\ \cite[Proposition 6.1]{Pav14}). The uniqueness of the invariant measure further implies the ergodicity (cf.\ \cite{Daprato06}), meaning that the temporal average converges to the
spatial average in the following sense:
\begin{equation*}
	\lim\limits_{T\to\infty}\frac{1}{T}\int_0^T \E\big[\phi\big(\bar{\mathfrak{P}}(t,p,q),\bar{\mathfrak{Q}}(t,p,q)\big)\big]\ud t=\int_{\R^{2d}}\phi(p,q)\widetilde{\mu}(\ud p\ud q)=:\widetilde{\mu}(\phi)~\text{in}~L^2(\mbb R^{2d};\widetilde{\mu}).
\end{equation*}
Here, $\big(\bar{\mathfrak{P}}(t,p,q),\bar{\mathfrak{Q}}(t,p,q)\big)$ denotes the solution of \eqref{eq:Langevinorg} at the time $t$ with the initial value $\big(\bar{\mathfrak{P}}(0),\bar{\mathfrak{Q}}(0)\big)=(p,q)$, and $\phi$ is any function belonging to $L^2(\mbb R^{2d};\widetilde{\mu})$. 
The ergodicity provides deep insight into the long-time dynamics and statistical properties of \eqref{eq:Langevinorg}, and underpins many sampling algorithms, including the Hybrid Monte Carlo method \cite{SanzSernaLNM} and its randomized variants \cite{Bou-RabeeAAP}. In practice, implementing these algorithms involves approximating \(\widetilde{\mu}\) or the ergodic limit \(\widetilde{\mu}(\phi)\), which is commonly achieved through carefully designed numerical discretizations of \eqref{eq:Langevinorg} over long-time horizons.

Owing to its broad applicability and physical richness, the numerical analysis of \eqref{eq:Langevinorg}, especially for long-time simulations, has attracted considerable attention. 
When the potential $V$ grows quadratically, there is already a substantial body of work on the numerical approximation of the invariant measure for \eqref{eq:Langevinorg} (cf.\ \cite{AVZ,AVZ15,BRO10,HS22,HSW17,MSH02} and references therein), where the convergence of numerical invariant measures usually relies on uniform-in-time strong or weak convergence analysis of numerical methods. 
In contrast, the analysis of uniform-in-time convergence and ergodicity of numerical methods becomes considerably more challenging in the case of superquadratic or general polynomially growing potentials such as the double-well potential $V(q) = \frac{1}{4}\|q\|^4 - \frac{1}{2}\|q\|^2,~q \in \mathbb{R}^d$. This is because  the drift coefficient $
b(p,q):=\big(-(\nabla V(q))^\top - v p^\top,\; p^\top\big)^\top
$ of \eqref{eq:Langevinorg} does not fulfill the one-sided Lipschitz condition or even the contractivity  at infinity (see \cite[Eq.\ (2.3)]{HuangLi25} for the definition of the contractivity  at infinity).
In this more demanding setting, to establish the existence of invariant measures and uniform-in-time weak error analysis, a key difficulty 
lies in how to design numerical methods satisfying a Lyapunov condition of the form
\begin{equation}\label{eq:FLC}
	\E[\mathcal{V}(X_{n+1})|\mathscr{F}_{t_n}]\le K_1 \mathcal{V}(X_n)+K_2\tau,~ n=0,1,2,\ldots
\end{equation}
for two constants $K_1\in(0,1)$ and $K_2>0.$
Here, $X_n$ denotes the numerical solution of  $(\bar{\mathfrak{P}}(n\tau),\bar{\mathfrak{Q}}(n\tau))$, $\mathcal{V}$ is a Lyapunov function, $\tau>0$ is the time stepsize, and $t_n=n\tau$. 
For \eqref{eq:Langevinorg} with superquadratic potentials, only implicit schemes have been shown to satisfy such a condition, including the implicit Euler method \cite{T02}, a splitting backward Euler scheme \cite{MSH02}, and a class of splitting methods based on invariant-preserving discretizations \cite{CDHZ25}.
Despite the superior stability of  these implicit methods, explicit schemes are generally preferred in stochastic computing, especially for long-term simulations, due to their simple implementation and lower computational cost. Motivated by this practical advantage, the authors in \cite{DaiIMA,HutzenthalerMC} proposed and analyzed stopped increment-tamed schemes for \eqref{eq:Langevinorg}, whose strong convergence rate in the finite time horizon was established via the exponential integrability of numerical solutions.
Recently, \cite{DJW25} proposed an explicit  splitting scalar auxiliary variable (SSAV) method, and obtained the weak convergence rate with error constant polynomially depending on the time length. Nevertheless,
to the best of our knowledge, no explicit methods have been shown to converge uniformly in time 
to \eqref{eq:Langevinorg}, for the case of polynomially growing potential $V$. This work is devoted to filling this gap and to developing explicit schemes that effectively approximate the invariant measure $\tilde{\mu}$. 

\subsection{A novel explicit numerical method}
In order to ensure uniform-in-time weak convergence, we  propose the following explicit numerical method for \eqref{eq:Langevinorg}:
\begin{align}\label{explicitmethod}
	\bar{\mathfrak{P}}_n:=P_n-\frac{v}{2}Q_n,~\bar{\mathfrak{Q}}_n:=Q_n,~n=0,1,\ldots,
\end{align}
where the sequence $\{(P_n,Q_n)\}_{n\ge0}$ is defined recursively by
\begin{subequations}\label{eq:SAVN}
	\begin{align}\label{eq:SAV1P}
		\frac{P_{n+1}-P_n}{\tau}=&\;\frac{r_{n+1}}{\sqrt{U_\alpha(Q_n)}}\left(-\nabla V(Q_n)+\alpha Q_n\right)-\frac{v}{2}P_{n+1}\\
		&\;+(\frac{v^2}{4} -\alpha)Q_{n+1}+\tau^{-1} \Delta W_n,\notag\\
		\frac{Q_{n+1}-Q_n}{\tau}=&\;P_{n+1}-\frac{v}{2} Q_{n+1},\label{eq:SAV1Q}\\
		\frac{r_{n+1}-r_n}{\tau}=&\;\frac{1}{2\sqrt{U_\alpha(Q_n)}}\langle\nabla V(Q_n)-\alpha Q_n,P_{n+1}\rangle-\frac{v}{4}\beta_0r_{n+1}-v\zeta_\alpha(Q_n)\label{eq:SAV1R}
	\end{align}
\end{subequations}
with $P_0-\frac{v}{2}Q_0=\bar{\mathfrak{P}}(0)$, $Q_0=\bar{\mathfrak{Q}}(0)$, and $r_0=\sqrt{U_\alpha(Q_0)}$.
Here, $\Delta W_n:=W(t_{n+1})-W(t_n)$, $	U_\alpha(q):=V(q)-\frac{\alpha}{2}\|q\|^2+C_{V}$ for some $\alpha>\frac{v^2}{4}$ and $C_V\ge 0$, and $\zeta_\alpha$ is defined by \eqref{eq:zetaalpha}.

The intermediate numerical  solution $\{(P_n,Q_n,r_n)\}_{n\ge 0}$ is constructed by applying the linearly implicit method to \eqref{eq:SAV1}, an extended system  with a scalar auxiliary variable (SAV).   Different from the existing SAV technique,  we design a dissipative term $-\frac{v}{4}\beta_0r_{n+1}$ in the  equation  \eqref{eq:SAV1R} by leveraging the polynomial algebraic structure of the potential $V$ (see Condition \ref{condition2}).  This dissipative term plays a key role in ensuring the Lyapunov condition \eqref{eq:FLC} and the uniform-in-time weak convergence.
Thus, we call the proposed numerical method \eqref{explicitmethod}-\eqref{eq:SAVN} the dissipative SAV (DSAV) method for \eqref{eq:Langevinorg}.  
To maintain the continuity of the main text, we defer the detailed construction procedure of the DSAV method and  its explicit solvability to Section \ref{Sec:construction}. 

\subsection{Main results}
Our first main result establishes the uniform-in-time first-order  weak convergence for the DSAV method \eqref{explicitmethod}-\eqref{eq:SAVN}. 
\begin{theo}\label{theo:weakorder}
	Let Assumptions  \ref{asp:V1} and \ref{asp:V0} hold and $(p,q)\in\mbb R^{2d}$ be given. Consider the numerical solution $\{(\bar{\mathfrak{P}}_n,\bar{\mathfrak{Q}}_n)\}_{n\ge 0}$ generated by the DSAV method \eqref{explicitmethod}-\eqref{eq:SAVN}   with $(P_0,Q_0,r_0)=(p+\frac{v}{2}q,q,\sqrt{U_\alpha(q)})$. Then for any $\varphi\in C^\infty_{\rm p}(\mbb R^{2d};\mbb R)$, there is a constant $C>0$ independent of $\tau$ such that
	\begin{align*}
		\sup_{n\in\mbb N}\big|\E[\varphi(\bar{\mathfrak{P}}_n,\bar{\mathfrak{Q}}_n)]-\E[\varphi(\bar{\mathfrak P}(t_n,p,q),\bar{\mathfrak Q}(t_n,p,q))]\big|\le C\tau.
	\end{align*}
\end{theo}	
Our strategy for proving Theorem \ref{theo:weakorder} comprises two main components: (a) establishing the Lyapunov condition for the DSAV method \eqref{explicitmethod}-\eqref{eq:SAVN} (see Lemma \ref{lem:Lyapunov}), and (b) deriving a uniform-in-time error estimate between the numerical SAV $r_n$ and the square root $\sqrt{U_\alpha(Q_n)}$ of the numerical standardized potential (see Lemma \ref{lem:rn-sqrt}).
The introduction of the dissipative term $-\frac{v}{4}\beta_0r_{n+1}$ in \eqref{eq:SAV1R} plays a crucial role in the proofs of both Lemmas \ref{lem:Lyapunov} and \ref{lem:rn-sqrt}. This, in turn, enables the uniform-in-time weak convergence of $\{(\bar{\mathfrak{P}}_n,\bar{\mathfrak{Q}}_n)\}_{n \ge 0}$  via the Kolmogorov equation.

As an immediate result of Theorem \ref{theo:weakorder} and the ergodicity of \eqref{eq:Langevinorg} (see \eqref{eq:PQergodic}),  we further characterize the convergence rate for approximating the invariant measure $\widetilde{\mu}$ based on $\{(\bar{\mathfrak{P}}_n,\bar{\mathfrak{Q}}_n)\}_{n\ge 0}$.
\begin{cor}\label{theo:appromeasure1}
	Under the conditions of Theorem \ref{theo:weakorder},  for any $\phi\in C^\infty_{\rm p}(\mbb R^{2d};\mbb R)$, there exist  constants $K(\phi)>0$ and $\widetilde C>0$  independent of $\tau$ and $n$ such that for any $n\in\mathbb{N}^+$,
	\begin{align}\label{sec5eq1}
		\Big|\E[\phi(\bar{\mathfrak{P}}_n,\bar{\mathfrak{Q}}_n)]-\widetilde{\mu}(\phi)\Big|\le \widetilde C(e^{-K(\phi)t_n}+\tau).
	\end{align}
\end{cor}
\noindent 
Corollary \ref{theo:appromeasure1} shows that to approximate $\widetilde{\mu}(\phi)$ with a precision $\varepsilon$, the necessary number of iteration steps is $t_n\tau^{-1}=\mcal O(\frac{1}{\varepsilon}\log \big(\frac{1}{\varepsilon}\big))$, which realizes a lower computational cost compared to $\mcal O\big(\frac{1}{\varepsilon}(\log \big(\frac{1}{\varepsilon}\big))^{1+l}\big)$ for some $l>0$ when applying the SSAV method to approximate $\widetilde{\mu}(\phi)$ (see \cite[Eq.\ (5.32)]{DJW25} for more details).

Our second main result gives the existence and positivity of the density function of the numerical solution $\{(\bar{\mathfrak{P}}_n,\bar{\mathfrak{Q}}_n)\}_{n\ge 0}$ for \eqref{eq:Langevinorg}.
\begin{theo}\label{thm:density}       
	Let Assumption \ref{asp:V1} hold. Then for any $n\ge 2$, the random variable $(P_{n},Q_{n})$ admits a density function, which is $Leb(\mbb R^{2d})$-a.e. positive.  Here, $Leb(\mbb R^{2d})$ denotes the Lebesgue measure on $\mbb R^{2d}$.
\end{theo} 
It is well known that the underdamped Langevin equation \eqref{eq:Langevinorg} is a typical stochastic dynamical system satisfying the parabolic H\"ormander condition \cite{Pav14}, whose exact solution $(\bar{\mathfrak{P}}(t), \bar{\mathfrak{Q}}(t))$, $t>0$,   admits a smooth and strictly positive density function. For the regularity of the density function of numerical solutions, it was shown in \cite{CHS19} that the splitting averaged vector field method for (1.1) admits a smooth density. The analysis therein relies on Malliavin calculus, a fundamental tool for studying the regularity of density functions of random variables (see, e.g., \cite{Nualart2006} and the references therein). We would like to mention that \cite{CHS19} also established the convergence rate of the density function of the numerical solution.  In contrast, the numerical scheme \eqref{eq:SAVN} proposed in this paper is constructed for the extended system \eqref{eq:SAV1}, which does not satisfy the parabolic H\"ormander condition (see Remark \ref{rem:exact-PDF}). Consequently, our proof of Theorem 1.4 follows a different approach, based on the invertibility of the diffusion coefficient together with a change-of-variables argument. A combination of \eqref{sec5eq1} and Theorem \ref{thm:density} further shows that the density of $(\bar{\mathfrak{P}}_n,\bar{\mathfrak{Q}}_n)$ provides a consistent approximation to $\widetilde{\rho}$, the density of the invariant measure $\widetilde{\mu}$; see Corollary \ref{cor:appromeasure2}.

With independent interest, we also study the exponential integrability of both the exact solution to \eqref{eq:Langevinorg} and the associated numerical solution to the DSAV method \eqref{explicitmethod}-\eqref{eq:SAVN}, which plays a crucial role in the analysis of strong convergence for the numerical solution to (1.1) (see, e.g., \cite{CDHZ25,CHS19}). Compared with existing results, we establish uniform-in-time upper bounds for the exponential moments of both the exact solution and the numerical solution in Lemmas \ref{expintergal} and \ref{Numer-expinter}, respectively. These bounds have potential applications to sharpen the error constants in strong convergence estimates for numerical schemes, as well as to improve tail probability estimates for both the exact and numerical solutions.  The uniform-in-time exponential integrability also implies the long-time moment stability of the DSAV method; this moment stability property is illustrated by the numerical experiments in Section \ref{Sec:num-exp} (see Figure \ref{F:moment}).

\subsection{Contributions} 
For the sake of clarity, we summarize the main contributions of this work as follows:
\begin{itemize}
	\item We develop a novel explicit numerical method, called the DSAV method, for the long-time simulation of \eqref{eq:Langevinorg}. The proposed method not only builds upon the existing SAV technique (see, e.g., \cite{DJW25}), but also integrates new design mechanisms to guarantee the Lyapunov condition \eqref{eq:FLC}. 
	
	\item To the best of our knowledge, this work provides the first analysis of uniform-in-time exponential integrability and uniform-in-time weak convergence for explicit numerical methods applied to \eqref{eq:Langevinorg} with polynomially growing potentials.

	\item We establish the existence and positivity of the density function of the numerical solution of $\{(\bar{\mathfrak{P}}_n,\bar{\mathfrak{Q}}_n)\}_{n\ge 0}$ based on a new argument which avoids reliance on the Malliavin Calculus. 
\end{itemize}

\emph{Organization.}
The paper is organized as follows. The detailed construction idea of the DSAV method is illustrated in Section \ref{Sec:construction}.
Section \ref{Sec:extended-system} investigates the exponential integrability and ergodicity of the extended system \eqref{eq:SAV1}. Section \ref{Sec:weakconvergence} and Section \ref{Sec:aapromeasure} are devoted to the proofs of Theorem \ref{theo:weakorder} and Theorem \ref{thm:density}, respectively. Numerical experiments are reported in Section \ref{Sec:num-exp} to validate the theoretical results. Finally, Section \ref{Sec:conclusion} concludes the paper with a summary and further remarks.

\section{Construction  procedure of DSAV method} \label{Sec:construction}
The construction of the DSAV method for \eqref{eq:Langevinorg} is inspired by the SAV technique, but incorporating a key design to guarantee the Lyapunov condition \eqref{eq:FLC} and achieve the uniform-in-time weak convergence. The SAV method, first proposed by \cite{ShenJCP,Shen18}, is suitable for solving deterministic gradient flows, enjoying the advantages of being unconditionally energy stable and computationally efficient as a linear-implicit scheme.  In the past several years, this method has been successfully adapted to some stochastic partial differential equations such as  stochastic  Allen--Cahn equations \cite{Metzger25,MetzgerESAIM}. Recently, \cite{DJW25} proposed a SSAV scheme for \eqref{eq:Langevinorg} by constructing a new SAV different from the existing literature, which attains the exponential  integrability in the finite time horizon and thus recovers the strong convergence rate.
However, the SAV method proposed in \cite{DJW25} does not appear to satisfy the Lyapunov condition, thereby hindering  the uniform-in-time weak convergence. 
To overcome this difficulty, we introduce an extra dissipative term on the SAV by leveraging the polynomial algebraic structure of the potential $V$. Next, we elaborate the construction idea of our proposed method.

First, by introducing the new variable $\mathfrak{P}(t)= \bar{\mathfrak{P}}(t)+\frac{v}{2} \bar{\mathfrak{Q}}(t)$ and denoting $\mathfrak{Q}(t)=\bar{\mathfrak{Q}}(t)$, we reformulate \eqref{eq:Langevinorg} into
\begin{equation}\label{eq:Langevin}
	\left\{
	\begin{split}
		\ud \mathfrak{P}(t)&=-\nabla V(\mathfrak{Q}(t))\ud t-\frac{v}{2}{}\mathfrak{P}(t)\ud t+\frac{v^2}{4} \mathfrak{Q}(t)\ud t+ \ud W(t),\\
		\ud \mathfrak{Q}(t)&=\mathfrak{P}(t)\ud t-\frac{v}{2}\mathfrak{Q}(t)\ud t.
	\end{split}
	\right.
\end{equation}
We point out that, although this transformation formally introduces a dissipative term for the position variable $\mathfrak{Q}$, it is not essential for the construction of the numerical method and serves primarily to simplify the notation.

Then inspired by \cite{DJW25}, we pose the following condition.
\begin{con}\label{condition1}
	There exist constants $\alpha>\frac{v^2}{4}$ and $C_V\ge 0$ such that 
	\begin{align}\label{Ualpha}
		U_\alpha(q):=V(q)-\frac{\alpha}{2}\|q\|^2+C_{V}\ge1,\quad q\in\R^d.
	\end{align}
\end{con}
Further, we introduce the SAV $r(t) := \sqrt{U_\alpha(\mathfrak{Q}(t))}$  to deal with the polynomial nonlinearity $\nabla V$ explicitly, and we call $U_\alpha$ the standardized potential. As stated in \cite[Remark 2.5]{DJW25}, $r(t)$ incorporates an additional term $-\frac{\alpha}{2}\|q\|^2$
compared with usual SAV like $\sqrt{ V(\mathfrak{Q}(t))+C_V}$, to ensure the exponential integrability and strong convergence rate in the finite time horizon. 
By adding the SAV $r(t)$,  \eqref{eq:Langevin} is turned into 
\begin{equation}\label{SAV-wang}
	\begin{split}
		\ud \mathfrak P(t)&=\frac{r(t)}{\sqrt{U_\alpha(\mathfrak Q(t))}}(-\nabla V(\mathfrak Q(t))+\alpha \mathfrak Q(t))\ud t-\frac{v}{2}\mathfrak P(t)\ud t+(\frac{v^2}{4} \!-\alpha)\mathfrak Q(t)\ud t+\ud W(t),\\
		\ud \mathfrak Q(t)&=\mathfrak P(t)\ud t-\frac{v}{2} \mathfrak Q(t)\ud t,\\
		\ud r(t)&=\frac{1}{2\sqrt{U_\alpha(\mathfrak Q(t))}}\langle\nabla V(\mathfrak Q(t))-\alpha \mathfrak Q(t),\mathfrak P(t)-\frac{v}{2}\mathfrak Q(t)\rangle\ud t.
	\end{split}
\end{equation}

Remarkably, it will be considerably difficult to obtain a uniform-in-time weakly convergent method, if one  considers  a direct discretization to \eqref{SAV-wang}.  This issue arises from the lack of manifest dissipativity for $r(t)$ in \eqref{SAV-wang}. To address this issue, we note the following algebraic structure shared by many polynomial potentials, including the double-well potential $V(q) = \frac{1}{4}\|q\|^4 - \frac{1}{2}\|q\|^2,~q \in \mathbb{R}^d$.
\begin{con}\label{condition2}
	There exist $\beta_0>0$ and a function $V_1:\mbb R^d\to\mbb R$ such that
	\begin{equation}\label{qVq}
		\langle q, \nabla V(q)\rangle= \beta_0V(q)+V_1(q),\quad q\in\mbb R^d.
	\end{equation}
\end{con}
Employing  Condition \ref{condition2}, $V(q)=U_\alpha(q)+\frac{\alpha}{2}\|q\|^2-C_V$, and $r(t)=\sqrt{U_\alpha(\mathfrak{Q}(t))}$, we arrive at
\begin{align*}
	\frac{\ud}{\ud t} r(t)
	&=\frac{1}{2\sqrt{U_\alpha(\mathfrak{Q}(t))}}\langle\nabla V(\mathfrak{Q}(t))-\alpha \mathfrak{Q}(t),\mathfrak{P}(t)\rangle-\frac{v}{4}\beta_0r(t)\\
	&\quad-\frac{v}{4\sqrt{U_\alpha(\mathfrak{Q}(t))}}[-C_{V}\beta_0+V_1(\mathfrak{Q}(t))-\alpha(1-\frac{1}{2}\beta_0)\|\mathfrak{Q}(t)\|^2].
\end{align*}
Plugging the above formula into \eqref{SAV-wang},  we have that $\big\{\big(\mathfrak{P}(t),\mathfrak{Q}(t),r(t)\big)\big\}_{t\ge 0}$ solves the following extended system:

\begin{equation}\label{eq:SAV1}
	\!\!\!\!	\left\{
	\begin{split}
		\ud P(t)&=\frac{r(t)}{\sqrt{U_\alpha(Q(t))}}(-\nabla V(Q(t))+\alpha Q(t))\ud t-\frac{v}{2}P(t)\ud t+(\frac{v^2}{4} \!-\alpha)Q(t)\ud t+\ud W(t),\\
		\ud Q(t)&=P(t)\ud t-\frac{v}{2} Q(t)\ud t,\\
		\ud r(t)&=\frac{1}{2\sqrt{U_\alpha(Q(t))}}\langle\nabla V(Q(t))-\alpha Q(t),P(t)\rangle\ud t-\frac{v}{4}\beta_0r(t)\ud t-v\zeta_\alpha(Q(t))\ud t,
	\end{split}
	\right.
\end{equation}
where $\zeta_\alpha$ is defined by
\begin{align}\label{eq:zetaalpha}
	\zeta_\alpha(q):=\frac{1}{4\sqrt{U_\alpha(q)}}[-C_{V}\beta_0+V_1(q)-\alpha(1-\frac{1}{2}\beta_0)\|q\|^2],~q\in\R^d.
\end{align}
Applying the linearly implicit scheme to \eqref{eq:SAV1} leads to \eqref{eq:SAVN}, which together with \eqref{explicitmethod} completes the construction of the DSAV method for \eqref{eq:Langevinorg}.

We close this section by presenting the explicit iterative formulation of the DSAV method \eqref{explicitmethod}-\eqref{eq:SAVN}.
Indeed, by \eqref{eq:SAV1Q},
\begin{equation}\label{eq:Qn+1}
	Q_{n+1}=\left(1+\frac{v}{2}\tau\right)^{-1}(Q_n+\tau P_{n+1}).
\end{equation}
Inserting \eqref{eq:Qn+1} into \eqref{eq:SAV1P} yields
\begin{align}\label{eq:Pn+1}
	P_{n+1}=f_{v,\tau}P_n+\tau f_{v,\tau}&\Big[\frac{r_{n+1}}{\sqrt{U_\alpha(Q_n)}}\left(-\nabla V(Q_n)+\alpha Q_n\right)\\\notag
	&\quad+(\frac{v^2}{4} -\alpha)(1+\frac{v}{2}\tau)^{-1}Q_n+\tau^{-1} \Delta W_n\Big],
\end{align}
where $f_{v,\tau}:=(1+\frac{v}{2}\tau)(1+v\tau+\alpha\tau^2)^{-1}$. 
Plugging \eqref{eq:Pn+1} into \eqref{eq:SAV1R}, we have that $r^{n+1}$ is solved explicitly via
\begin{align}\label{eq:rn+1-old}
	r_{n+1}&=G_0(Q_n)^{-1}r_n+G_0(Q_n)^{-1}G_1(Q_n,P_n)+G_0(Q_n)^{-1}\langle \tau G_2(Q_n),\Delta W_n\rangle.
\end{align}
where
\begin{gather}
	G_0(q):=1+\frac{v}{4}\tau\beta_0+\frac{\tau^2f_{v,\tau}}{2U_\alpha(q)} \|\nabla V(q)-\alpha q\|^2,\label{eq:G0}\\\notag
	G_1(q,p):=\frac{\tau}{2\sqrt{U_\alpha(q)}}\big\langle\nabla V(q)-\alpha q,f_{v,\tau}p+\tau f_{v,\tau}(\frac{v^2}{4} -\alpha)(1+\frac{v}{2}\tau)^{-1}q\big\rangle-v\tau\zeta_\alpha(q),
	\\
	G_2(q):=\frac{ f_{v,\tau}}{2\sqrt{U_\alpha(q)}}(\nabla V(q)-\alpha q).\label{eq:G2}
\end{gather}
Then $P_{n+1}$ and $Q_{n+1}$ can be solved from \eqref{eq:Pn+1} and \eqref{eq:Qn+1}, respectively.

\section{Exponential integrability and ergodicity of extended system} \label{Sec:extended-system}
In this section, we present the uniform-in-time exponential integrability and ergodicity of the extended system \eqref{eq:SAV1}. 

We begin with some notation. Denote by $\|\cdot\|$ the trace norm of a vector or matrix, i.e., $\|A\|=\sqrt{\mathrm{tr}(A^\top A)}$. Let $\LL\cdot,\cdot\RR$ denote the scalar product of two vectors. Let $\mscr B_b(\mbb R^{d};\mbb R)$ stand for the set of bounded measurable functions from $\mbb R^{d}$ to $\mbb R$. Denote by $ C_b(\mathbb R^d;\mbb R^m)$ (resp.\ $ C^{k}(\mathbb R^d;\mbb R^m)$) the space of bounded and continuous (resp.\ $k$th continuously differentiable) functions from $\mathbb R^d$ to $\mbb R^m$. 
For $f\in C^k(\mathbb R^d;\mbb R)$, denote by $ D^k f(x)(\xi_1,\ldots,\xi_k)$ the $k$th order G\^ ateaux derivative along the directions $\xi_1,\ldots,\xi_k\,\in\mathbb R^d$. 
Let $\mathbf F$ be the set of functions with polynomial growth at infinity, i.e., a function $f\in\mathbf F$ means that there exist $C>0$ and $l>0$ such that for any $x\in\mathbb R^d$, $|f(x)|\le C(1+\|x\|^l)$. Then denote by $C^\infty_{\rm p}(\mbb R^d;\mbb R)$ the set of functions $f\in C^\infty(\mbb R^d;\mbb R)$ whose all partial derivatives belong to $\mbf F$.
We assume that the Brownian motion $W$ is defined on a complete filtered probaility space $\big(\Omega,\mcal F,\{\mcal F_t\}_{t\ge 0},\mbb P\big)$. 
Throughout this paper, let $K(a_1,a_2,...,a_m)$ be a generic constant dependent on parameters $a_1,a_2,...,a_m$ but independent of the stepsize $\tau$ and the step number $n$, which may vary for each appearance.

\subsection{Exponential integrability}
To ensure the well-posedness of the extended system \eqref{eq:SAV1}, we introduce the following assumption. 
\begin{assum}\label{asp:V1}
	Let Conditions \ref{condition1} and \ref{condition2} hold. Moreover, assume that $\zeta_\alpha$ defined in \eqref{eq:zetaalpha} satisfies that for any $q\in\R^d$,
		$|\zeta_\alpha(q)|\le L_1(\|q\|^{\beta_1}+1)$, 
	where	$\beta_1\in(0,1)$ and $L_1>0$ are two constants.
\end{assum}
\noindent Under Assumption \ref{asp:V1}, for any $\epsilon_0>0$,
\begin{equation}\label{eq:zeta}
	\zeta^2_\alpha(q)\le \epsilon_0\|q\|^2+K(\epsilon_0,L_1,\beta_1),\quad q\in\mbb R^d.
\end{equation}
\vspace{-3mm}
\begin{rem}\label{rem:example}
	We remark that the condition $V\ge 0$ in Assumption \ref{asp:V0} can be replaced by the one that $V$ is bounded below. In addition, Assumptions \ref{asp:V1} and \ref{asp:V0}  are fulfilled by a large class of potential functions $V$, which are essentially of polynomial growth. An example is
	$V(q)=a_0\|q\|^{2l}+F(q)+G(q),~ q\in\mbb R^{d},$ including the double-well potential, where $l\ge 2$ is a positive integer, $a_0>0$, $F$ is a  polynomial of degree at most $l$, and $G\in C^{\infty}(\mbb R^d;\mbb R)$ is any function of linear growth.
\end{rem}

For the extended system \eqref{eq:SAV1}, we introduce the Lyapunov function
\begin{align}
	H(p,q,r)=\frac12\|p\|^2+\frac12(\alpha-\frac{v^2}{4})\|q\|^2+r^2,\quad p,q\in\R^d, r\in\R.
\end{align}
\noindent Denote by $\mathcal{L}_{\textup{ext}}$ the infinitesimal generator of \eqref{eq:SAV1}, i.e.,
\begin{align*}
	&\;\mathcal{L}_{\textup{ext}}g(p,q,r)\\
	=&\;\nabla_p g(p,q,r)^\top\Big[\frac{r}{\sqrt{U_\alpha(q)}}(-\nabla V(q)+\alpha q)-\frac{v}{2}p+(\frac{v^2}{4}-\alpha)q\Big]+\nabla_q g(p,q,r)^\top(p-\frac{v}{2}q)\\
	&\;+\frac{\partial g(p,q,r)}{\partial r}\Big[\frac{1}{2\sqrt{U_\alpha(q)}}\langle \nabla V(q)-\alpha q,p\rangle-\frac{v}{4}\beta_0r-v\zeta_\alpha(q)\Big]+\frac{1}{2}\textup{tr}\big(\nabla^2_{pp}g(p,q,r)\big)
\end{align*}
for any $g\in C^2(\mbb R^d\times\mbb R^d\times\mbb R;\mbb R)$,
where $\nabla^2_{pp}g(p,q,r)=\big(\frac{\PD^2 g(p,q,r)}{\PD p_i\PD p_j}\big)_{i,j=1,\ldots,d}$ is the Hessian matrix of $g$ with respect to $p$. 
Then by Young's inequality and \eqref{eq:zeta}, for $\beta_0^{\prime}:=\min\{\beta_0,2\}$,
\begin{align*}
	\mathcal{L}_{\textup{ext}}H(p,q,r)&=-\frac{v}{2}\left(\|p\|^2+\beta_0r^2+(\alpha-\frac{v^2}{4})\|q\|^2\right)-2rv\zeta_\alpha(q)+\frac{d}{2}\\
	&\le -\frac{v}{2}\beta_0^{\prime}H(p,q,r)+\frac{\beta_0^{\prime}}{4}vr^2+v\frac{4}{\beta_0^{\prime}}\left(\epsilon_0\|q\|^2+K(\epsilon_0,L_1,\beta_1)\right)+\frac{d}{2}.
\end{align*}
Choosing $0<\epsilon_0\ll 1$ such that for any $q\in\R^d$,
$\frac{4}{\beta_0^{\prime}}\epsilon_0\|q\|^2\le \frac{1}{8}\beta_0^{\prime}(\alpha-\frac{v^2}{4})\|q\|^2\le \frac{1}{4}\beta_0^{\prime}H(p,q,r)-\frac{\beta_0^{\prime}}{4}r^2.$
As a result, there exists a positive constant $C_0$ such that for any $p,q\in\R^{d}$ and $r\in\R$,
\begin{equation}\label{eq:Lext}
	\mathcal{L}_{\textup{ext}}H(p,q,r)
	\le -\frac{v}{4}\beta_0^{\prime}H(p,q,r)+C_0,
\end{equation}
which implies that $\mathcal{L}_{\textup{ext}}(H+1)\le C_0 (H+1)$. It follows from \cite[Theorem 10.2]{YZH17} that
the extended system \eqref{eq:SAV1} admits a unique strong solution for any non-random initial value $(P(0),Q(0),r(0))$. 

Let $(P(t,p,q,r_0),Q(t,p,q,r_0),r(t,p,q,r_0))$ (resp.\ $(\mathfrak{P}(t,p,q),\mathfrak{Q}(t,p,q))$) denote \\ the solution of \eqref{eq:SAV1} (resp.\ \eqref{eq:Langevin}) with the initial value $(p,q,r_0)$ (resp.\ $(p,q)$). The following fact further demonstrates the link between \eqref{eq:Langevin} and \eqref{eq:SAV1}.
\begin{lem}\label{equivalent}
	Let Assumption \ref{asp:V1} hold. If a process $\big\{\big(P(t),Q(t),r(t)\big)\big\}_{t\ge 0}$ solves \eqref{eq:SAV1} with the initial value $(p,q,r_0)$ satisfying $r_0=\sqrt{U_\alpha(q)}$, then for any $t\ge 0$, it holds that $r(t)=\sqrt{U_\alpha(Q(t))}$, and that
	$\{( P(t), Q(t))\}_{t\ge 0}$ solves \eqref{eq:Langevin}. Thus, 
	$$P(t,p,q,\sqrt{U_\alpha(q)})=\mathfrak{P}(t,p,q),~Q(t,p,q,\sqrt{U_\alpha(q)})=\mathfrak{Q}(t,p,q),\quad (p,q)\in\mbb R^{2d}.$$
\end{lem}
\begin{proof}
	Owing to the derivation of \eqref{eq:SAV1}, it follows that
	$\{(\mathfrak{P}(t,p,q),\mathfrak{Q}(t,p,q),\\\sqrt{U_\alpha(\mathfrak{Q}(t,p,q))})\}_{t\ge0}$ solves \eqref{eq:SAV1}. The uniqueness of the strong solution of \eqref{eq:SAV1} gives that $(P(t,p,q,r_0),Q(t,p,q,r_0),r(t,p,q,r_0))=(\mathfrak{P}(t,p,q),\mathfrak{Q}(t,p,q),\sqrt{U_\alpha(\mathfrak{Q}(t,p,q)})$ \\for any $t\ge 0$, provided that $r_0=\sqrt{U_\alpha(q)}$. 
	Thus, the proof is complete.
\end{proof}

The following lemma states that the solution to the extended system \eqref{eq:SAV1} possesses the uniform-in-time exponential integrability. The Lyapunov structure \eqref{eq:Lext} plays a key role in its proof. Hereafter, we always assume that the initial values of \eqref{eq:SAV1} and \eqref{eq:SAVN} are non-random for simplicity.
\begin{lem}\label{expintergal}
	Let Assumption \ref{asp:V1} hold. Then for any $t\ge 0$,
	\begin{align*}
		\E\left[\exp\left(\kappa_1 H(P(t),Q(t),r(t))\right)\right]\le \exp\left(\kappa_1 H(P(0),Q(0),r(0))\right)+2,
	\end{align*}
	where $\kappa_1:=v\beta_0'\min\{\frac{1}{8},\frac{1}{16C_0}\}>0$ with $\beta'_0=\min\{\beta_0,2\}$ and $C_0$ given in \eqref{eq:Lext}.
\end{lem}
\begin{proof}
	Since $0<\kappa_1\le \frac{v}{8}\beta_0^{\prime}$ and $H(p,q,r)\ge \frac{1}{2}\|p\|^2$, we deduce from \eqref{eq:Lext} that
	\begin{align}
		\kappa_1\mathcal{L}_{\textup{ext}} H(p,q,r)+\frac{\kappa_1^2}{2}\| p\|^2\le 
		-\frac{v}{8}\beta_0^{\prime}\kappa_1H(p,q,r)+\kappa_1 C_0. \label{sec3eq3}
	\end{align}
	By
	$\nabla^2_{pp}e^{\kappa_1H(p,q,r)}=
	\kappa_1	e^{\kappa_1H(p,q,r)}(\kappa_1\nabla_p H(p,q,r)\nabla_p H(p,q,r)^\top +\nabla^2_{pp}H(p,q,r))$
	and
	It\^o's formula, and denoting $X(t):=(P(t),Q(t),r(t))$, we have that for any $\lambda\in\mbb R$,
	\begin{align*}
		\ud e^{\lambda t +\kappa_1 H(X(t))}&=e^{\lambda t+\kappa_1 H(X(t))}\Big[\kappa_1\mathcal{L}_{\textup{ext}} H(X(t)))+\frac{\kappa_1^2}{2}\| P(t)\|^2+\lambda\Big]\ud t\\
		&\quad+\kappa_1e^{\lambda t+\kappa_1 H(X(t))}P(t)^\top\ud W(t),\quad t\ge 0.
	\end{align*}
	Denote $\eta_R:=\inf\{t>0:\|X(t)\|>R\}$, $R>1$. Then for any $t\ge 0$, $R>1$, and $\lambda\in\mbb R$,
	\begin{align*}
		&\mbb Ee^{\lambda (t\wedge \eta_R)+\kappa_1 H(X(t\wedge \eta_R))}
		= e^{\kappa_1 H(X(0))}\\
		&+\mbb E\int_0^{t\wedge \eta_R}e^{\lambda s+\kappa_1 H(X(s))}\Big[\kappa_1\mathcal{L}_{\textup{ext}} H(X(s)))+\frac{\kappa_1^2}{2}\| P(s)\|^2+\lambda \Big]\ud s,
	\end{align*}
	which, together with \eqref{sec3eq3} and $xe^x\ge e^x -1$, $x\in\mbb R$, yields		
	\begin{align*}
		&\;\mbb Ee^{\lambda (t\wedge \eta_R)+\kappa_1 H(X(t\wedge \eta_R))} \\
		\le &\; e^{\kappa_1 H(X(0))}+\mbb E\int_0^{t\wedge \eta_R}e^{\lambda s+\kappa_1 H(X(s))}\Big[-\frac{v}{8}\beta_0'\kappa_1 H(X(s))+\kappa_1C_0+\lambda\Big]\ud s \\
		\le &\; e^{\kappa_1 H(X(0))}+\mbb E\int_0^{t\wedge \eta_R}e^{\lambda s}\big[(\kappa_1 C_0+\lambda-\frac{v}{8}\beta_0')e^{\kappa_1 H(X(s))}+\frac{v}{8}\beta_0'\big]\ud s.
	\end{align*}		
	Setting $\lambda=\lambda_0:=\frac{v\beta_0'}{16}$,
	then it holds $\kappa_1 C_0+\lambda_0-\frac{v}{8}\beta_0'\le 0$ due to $\kappa_1\le \frac{v}{16C_0}\beta_0^{\prime}$. Accordingly,
	\begin{align}\label{sec3eq4}
		\mbb Ee^{\lambda_0 (t\wedge \eta_R)+\kappa_1 H(X(t\wedge \eta_R))}
		\le e^{\kappa_1 H(X(0))}+\frac{v\beta_0'}{8\lambda_0}(e^{\lambda_0 t}-1). 
	\end{align} 		
	Note that $\eta_R\uparrow +\infty$, a.s., which, combined with \eqref{sec3eq4} and the Fatou lemma, gives
	\begin{align*}
		\mbb Ee^{\lambda_0 t+\kappa_1 H(X(t))}\le e^{\kappa_1 H(X(0))}+\frac{v\beta_0'}{8\lambda_0}(e^{\lambda_0 t}-1)\quad\forall~t\ge 0.
	\end{align*} 
	Thus, one has that $\mbb Ee^{\kappa_1 H(X(t))}\le e^{\kappa_1 H(X(0))}+\frac{v\beta_0'}{8\lambda_0}$.
	The proof is complete.
\end{proof}

\begin{rem}
	It follows from Lemmas \ref{equivalent} and \ref{expintergal} that the solutions of \eqref{eq:Langevin} and \eqref{eq:Langevinorg} are also uniform-in-time exponentially integrable.
\end{rem}

\subsection{Invariant measure and ergodicity}
First, we introduce the following assumption ensuring the ergodicity of \eqref{eq:Langevinorg}.
\begin{assum}\label{asp:V0}\cite[Condition 3.1]{MSH02}
	Assume that $V\in C^\infty_{\rm p}(\R^d;\mbb R)$ and satisfies 
	(i) $V(q)\ge0$ for all $q\in\R^d$;
	(ii) There are $\alpha_0>0$ and $\alpha_1\in(0,1)$ such that 
	\begin{align*}
		\frac{1}{2}\langle \nabla V(q),q\rangle\ge \alpha_1V(q)+v^2\frac{\alpha_1(2-\alpha_1)}{8(1-\alpha_1)}\|q\|^2-\alpha_0.
	\end{align*}
\end{assum}
Define the Lyapunov function
\begin{align}\label{eq:H}
	\widetilde H(p,q)=\frac{1}{2}\|p\|^2+V(q)+\frac{v}{2}\langle p,q\rangle +\frac{v^2}{4}\|q\|^2+1,\quad p,q\in \R^d.
\end{align}
According to \cite[Theorem 3.2]{MSH02}, under Assumption \ref{asp:V0}, for any $l\ge1$, there exist positive numbers $K_l$ and $k_l$ such that for any measurable function $\phi:\R^{2d}\to \R$ with $|\phi(p,q)|\le (\widetilde H(p,q))^l$, 
\begin{equation}\label{eq:PQergodic}
	\Big|\E\big[\phi\big(\bar{\mathfrak{P}}(t,p,q),\bar{\mathfrak{Q}}(t,p,q)\big)\big] -
	\widetilde\mu(\phi)\Big|\le K_l (\widetilde H(p,q))^le^{-k_l t}\quad \forall~t\ge0.
\end{equation} 
\begin{rem}\label{remark1}
	Since $\widetilde{H}(p,q)\ge 1+\frac{1}{8}\|p\|^2+\frac{v^2}{12}\|q\|^2$, we have that \eqref{eq:PQergodic} holds for any $\phi\in C^\infty_{\rm p}(\mbb R^{2d};\mbb R)$.
\end{rem}

Denote by $\{S^{\textup{ext}}_t\}_{t\ge0}$ (resp.\ $\{S_t\}_{t\ge0}$) the Markov semigroup generated by \eqref{eq:SAV1} (resp.\ \eqref{eq:Langevin}), i.e.,
\begin{align*}
	S^{\textup{ext}}_t\psi(p,q,r_0)&=\mbb E\big[\psi\big(P(t,p,q,r_0),Q(t,p,q,r_0),r(t,p,q,r_0)\big)\big],\quad\psi\in\mscr B_b(\mbb R^{2d+1};\mbb R),\\
	S_t\phi(p,q)&=\mbb E\big[\phi\big(\mathfrak{P}(t,p,q),\mathfrak{Q}(t,p,q)\big)\big],\quad \phi\in\mscr B_b(\mbb R^{2d};\mbb R).
\end{align*} 
Due to the equivalence between \eqref{eq:Langevinorg} and \eqref{eq:Langevin}, the semigroup $\{S_t\}_{t\ge 0}$ admits a unique invariant measure, denoted by $\mu$, and thus $\{S_t\}_{t\ge 0}$ or \eqref{eq:Langevin} is also ergodic. 
\begin{lem}\label{lem:inv-measure-projection}
	Let Assumptions \ref{asp:V1} and \ref{asp:V0} hold.
	Define the map $\pi:\R^{2d}\to \mathcal{M}:=\{(p,q,r)\in\R^{2d+1}: r=\sqrt{U_\alpha(q)}\}$ by $\pi(p,q)=(p,q,\sqrt{U_\alpha(q)})$ and let 
	$\nu_0:=\pi_\#\mu$ be the pushforward of $\mu$ under $\pi$. Then $\nu_0$ is an ergodic invariant measure for the extended system \eqref{eq:SAV1}.
\end{lem}

\begin{proof}
	Since $\nu_0=\pi_\#\mu$ is supported on $\mathcal{M}$, for $g\in C_b(\R^{2d+1};\mbb R)$,
	\begin{align}\label{sec3eq1}
		\int_{\R^{2d+1}}g\ud\nu_0=\int_{\mcal M}g\ud\nu_0
		=\int_{\R^{2d}}g\circ\pi\ud\mu.
	\end{align}
	Applying Lemma \ref{equivalent},
	we obtain for any $\Phi\in C_b(\R^{2d+1};\mbb R)$ and $t\ge0$,
	\begin{align}\label{sec3eq2}
		&S_t^{\textup{ext}}\Phi(\pi(p,q))=\E^{}[\Phi(P(t,\pi(p,q)),Q(t,\pi(p,q)),r(t,\pi(p,q)))]\notag\\
		&=\E[\Phi(\mathfrak{P}(t,p,q),\mathfrak{Q}(t,p,q),\sqrt{U_\alpha(\mathfrak{Q}(t,p,q))})]=S_t(\Phi\circ \pi)(p,q).
	\end{align}
	Using \eqref{sec3eq1} twice, \eqref{sec3eq2}, and the fact that $\mu$ is invariant for $\{S_t\}_{t\ge 0}$, we infer that for any $\Phi\in C_b(\R^{2d+1};\mbb R)$ and $t\ge0$,
	\begin{equation*}
		\int_{\R^{2d+1}}S_t^{\textup{ext}}\Phi\ud\nu_0\!=\!\int_{\R^{2d}} (S_t^{\textup{ext}}\Phi)\circ\pi\ud\mu =\int_{\R^{2d}} S_t(\Phi\circ\pi)\ud\mu \!=\!\int_{\R^{2d}}(\Phi\circ\pi)\ud\mu=\int_{\R^{2d+1}}\Phi\ud\nu_0,
	\end{equation*}
	so $\nu_0$ is an invariant measure of the extended system \eqref{eq:SAV1} or $\{S^{\textup{ext}}_t\}_{t\ge0}$.
	
	Next, we prove the ergodicity of $\nu_0$, by virtue of the ergodicity of $\mu$ for $\{S_t\}_{t>0}$.
	Let $A\subset \mbb R^{2d+1}$ be an $\nu_0$-invariant set for the extended dynamics \eqref{eq:SAV1}, i.e., there is a Borel set $\Gamma\subset \mbb R^{2d+1}$ with $\nu_0(\Gamma)=1$ such that
	$S_t^{\textup{ext}}\mathbf{1}_A(p,q,r)=\mathbf{1}_A(p,q,r)$ for any $(p,q,r)\in\Gamma$ and $t\ge 0$. Setting $B:=\pi^{-1}(A)\subset\R^{2d}$, we have  $\mathbf{1}_B(p,q)=\mathbf{1}_A\circ \pi(p,q)$ for any $(p,q)\in\mbb R^{2d}$. 
	Since $C_b(\mbb R^{2d+1};\mbb R)$ is dense in $\mscr B_b(\mbb R^{2d+1};\mbb R)$ under the topology of pointwise convergence, \eqref{sec3eq2} also holds for any $g\in\mscr B_b(\mbb R^{2d+1};\mbb R)$. 
	Thus, it holds that for $(p,q)\in\pi^{-1}(\Gamma)$,
	$
	S_t\mathbf{1}_B(p,q)=S_t(\mathbf{1}_A\circ \pi)(p,q)
	=S_t^{\textup{ext}}\mathbf{1}_A(\pi(p,q))
	=\mathbf{1}_A(\pi(p,q))
	=\mathbf{1}_B(p,q),
	$
	which implies that $B$ is $\mu$-invariant for the original system \eqref{eq:Langevin} due to $\mu(\pi^{-1}(\Gamma))=\nu_0(\Gamma)=1$. 
	By the ergodicity of $\mu$ for \eqref{eq:Langevin}, we have $\mu(B)\in\{0,1\}$, which implies
	$
	\nu_0(A)=\mu(\pi^{-1}(A))=\mu(B)\in\{0,1\}.
	$
	Hence $\nu_0$ is ergodic owing to \cite[Theorem 5.15]{Daprato06}.
\end{proof}

\section{Proof of Theorem \ref{theo:weakorder}}\label{Sec:weakconvergence}
In this section, we make use of the Kolmogorov equation to give the proof of Theorem \ref{theo:weakorder}, where two key ingredients are  the uniform-in-time exponential integrability of the numerical solution $\{(P_n,Q_n,r_n)\}$ and uniform-in-time error estimate between $r_n$ and $\sqrt{U_\alpha(Q_n)}$.

\subsection{Uniform-in-time exponential integrability}\label{Sec:3.1}
For the numerical method \eqref{eq:SAVN} for \eqref{eq:SAV1}, we introduce the Lyapunov function
\begin{equation}
	\mathcal{V}^\tau(p,q,r):=\frac{1}{4}\|p\|^2+\frac{1}{2}r^2+\frac14(1+v\tau)(\alpha-\frac{v^2}{4})\left(1-\frac{v}{2(1+v\tau)}\tau\right)\|q\|^2,
\end{equation}
where $p,q\in \R^d$ and $r\in\R$.
For simplicity, we denote $\mathcal{V}^\tau_n:=\mathcal{V}^\tau(P_{n},Q_{n},r_{n})$ for any $n=0,1,\ldots$ Then the following Lyapunov structure holds.
\begin{lem}\label{lem:Lyapunov}
	Let Assumption \ref{asp:V1} hold. Denote $v_0:=\min\{\frac{v}{3},\frac{v}{4}\beta_0\}$. 
	Then there exists $K>0$ such that for any $\tau\in(0,v^{-1})$,
	$\E[\mathcal{V}^\tau_{n+1}|\mathscr{F}_{t_n}]\le (1-\frac12 v_0\tau) \mathcal{V}^\tau_n+K\tau.
	$
\end{lem}
\begin{proof}
	Owing to \eqref{eq:SAVN},
	\begin{align*}
		&\frac12\langle P_{n+1},\frac{P_{n+1}-P_n}{\tau}\rangle=\frac{r_{n+1}}{\sqrt{U_\alpha(Q_n)}}
		\langle \frac12P_{n+1},-\nabla V(Q_n)+\alpha Q_n\rangle-\frac{v}{4}\|P_{n+1}\|^2\\
		&\qquad\qquad\qquad\qquad\qquad+\frac12(\frac{v^2}{4} -\alpha)\langle P_{n+1},Q_{n+1}\rangle+\frac12\tau^{-1}\langle P_{n+1}, \Delta W_n\rangle,\\
		&r_{n+1}\frac{r_{n+1}-r_n}{\tau}
		=\frac{r_{n+1}}{\sqrt{U_\alpha(Q_n)}}\langle\nabla V(Q_n)-\alpha Q_n,\frac12 P_{n+1}\rangle-\frac{v}{4}r_{n+1}^2\beta_0-vr_{n+1}\zeta_\alpha(Q_n),\\
		&\frac12(\alpha-\frac{v^2}{4})\langle Q_{n+1},\frac{Q_{n+1}-Q_n}{\tau}\rangle
		=\frac12(\alpha-\frac{v^2}{4})\langle Q_{n+1},P_{n+1}\rangle-\frac{v}{4}(\alpha-\frac{v^2}{4})\|Q_{n+1}\|^2.
	\end{align*}
	Using the above three relations and
	the fundamental identity $2\langle a,a-b\rangle=\|a\|^2-\|b\|^2+\|a-b\|^2$ for $a,b\in\R^d$, we arrive at 
	\begin{align}\label{sec3eq5}
		&\frac14\left(\|P_{n+1}\|^2-\|P_n\|^2+\|P_{n+1}-P_n\|^2\right)+ \frac{1}{2}\left(r_{n+1}^2-r_n^2+|r_{n+1}-r_n|^2\right)\\
		&\quad+\frac14(\alpha-\frac{v^2}{4})\left(\|Q_{n+1}\|^2-\|Q_n\|^2+\|Q_{n+1}-Q_n\|^2\right)\notag\\
		&=\!-\frac{v}{4}\tau\|P_{n+1}\|^2\!+\!\frac12\langle P_{n+1}, \Delta W_n\rangle\!-\frac{v}{4}(\alpha-\frac{v^2}{4})\tau\|Q_{n+1}\|^2\!-\!\frac{v}{4}\tau \beta_0r_{n+1}^2\!-vr_{n+1}\tau\zeta_\alpha(Q_n). \notag
	\end{align}
	By \eqref{eq:zeta} and the Young inequality,
	\begin{equation}\label{sec3eq6}
		|vr_{n+1}\tau\zeta_\alpha(Q_n)|\le\frac{v}{8}\tau r_{n+1}^2\beta_0+\frac{2}{\beta_0}v\tau|\zeta_\alpha(Q_n)|^2\le \frac{v}{8}\tau r_{n+1}^2\beta_0+\frac{v}{8}(\alpha-\frac{v^2}{4})\tau\|Q_n\|^2+K\tau.
	\end{equation}
	Plugging \eqref{sec3eq6} into \eqref{sec3eq5} leads to
	\begin{align*}
		&\frac14\left(\|P_{n+1}\|^2-\|P_n\|^2+\|P_{n+1}-P_n\|^2\right)+ \frac{1}{2}\left(r_{n+1}^2-r_n^2+|r_{n+1}-r_n|^2\right)\\
		&\quad+\frac14(1+v\tau)(\alpha-\frac{v^2}{4})\left(\|Q_{n+1}\|^2-\|Q_n\|^2\right)+\frac14(\alpha-\frac{v^2}{4})\|Q_{n+1}-Q_n\|^2\\
		&\le-\frac{v}{4}\tau\|P_{n+1}\|^2+\frac18\|P_{n+1}-P_n\|^2+K\|\Delta W_n\|^2-\frac{v}{8}\beta_0\tau r_{n+1}^2\\
		&\quad-\frac{v}{8}(\alpha-\frac{v^2}{4})\tau\|Q_n\|^2+K\tau+\frac12\langle P_n, \Delta W_n\rangle.
	\end{align*}
	Rearranging the above inequality, we have 
	\begin{align*}
		&\frac{1}{4}(1+v\tau)\|P_{n+1}\|^2+\frac{1}{2}(1+\frac{v}{4}\beta_0\tau)r_{n+1}^2+\frac14(1+v\tau)(\alpha-\frac{v^2}{4})\|Q_{n+1}\|^2\\
		&\le \frac{1}{4}\|P_{n}\|^2+\frac{1}{2}r_n^2+\frac14(1+v\tau)(\alpha-\frac{v^2}{4})\left(1-\frac{v}{2(1+v\tau)}\tau\right)\|Q_{n}\|^2\\
		&\quad+K\|\Delta W_n\|^2+K\tau+\frac12\langle P_n, \Delta W_n\rangle.
	\end{align*}
	It follows from $v_0=\min\{\frac{v}{3},\frac{v}{4}\beta_0\}$ and $\tau\in(0,v^{-1}]$ that $(1+v_0\tau)^{-1}\le 1-\frac{v_0}{2}\tau$ and $1+v_0\tau\le \min\big\{1+v\tau,1+\frac{v}{4}\beta_0\tau,\big(1-\frac{v}{2(1+v\tau)}\tau\big)^{-1}\big\}.$ 
	Then
	there is a constant $C_1>0$ independent of $\tau$ such that for any $n\in\mathbb{N}$, 
	\begin{align}\label{eq:Vn+1}
		(1+v_0\tau)\mathcal{V}^\tau_{n+1}&\le\mathcal{V}^\tau_n+C_1\|\Delta W_n\|^2+C_1\tau+\frac12\langle P_n,\Delta W_n\rangle\quad\forall~\tau\in(0,v^{-1}].
	\end{align}
	Thus, the proof is completed.
\end{proof}
Next, we prove the uniform-in-time exponential integrability of the numerical solution $\{(P_n,Q_n,r_n)\}$. 
\begin{lem}\label{Numer-expinter}
	Let Assumption \ref{asp:V1} hold. Denote $\kappa_2:=\frac{v_0}{4}$ with $v_0=\min\{\frac{v}{3},\frac{v}{4}\beta_0\}$, and $\tau_1:=\min\{v^{-1},(8C_1\kappa_2)^{-1}\}$ with $C_1$ defined in \eqref{eq:Vn+1}.
	Then there exists a constant $K(P_0,Q_0,r_0)>0$ such that
	$\sup_{n\in \mathbb{N}}\E[e^{\kappa_2\mathcal{V}_{n+1}^\tau}]\le K(P_0,Q_0,r_0)$ for any $\tau\in(0,\tau_1].$
\end{lem}
\begin{proof}
	It follows from \eqref{eq:Vn+1} that for any $\lambda>0$ and $\tau\in(0,v^{-1}]$, 
	\begin{align}\label{sec4eq3}
		e^{\lambda\mathcal{V}_{n+1}^\tau}\le e^{\lambda(1+v_0\tau)^{-1}\mathcal{V}_{n}^\tau+C_1\lambda\|\Delta W_n\|^2+C_1\lambda\tau+\frac12(1+v_0\tau)^{-1}\lambda\langle P_n, \Delta W_n\rangle}.
	\end{align}
	Since $\Delta W_n\sim \mathcal{N}(0,\tau I)$, it holds that for any $\lambda>0$ with $2C_1\lambda<1/(2\tau)$,
	\begin{align*}
		\E\left[e^{2C_1\lambda\|\Delta W_n\|^2}\right]&=\prod_{i=1}^d\E\left[e^{2C_1\lambda\|\Delta W^i_n\|^2}\right]= (1-4C_1\lambda\tau)^{-\frac{d}{2}}\\
		&\le 	(1+8C_1\lambda\tau)^{\frac{d}{2}}\le e^{4dC_1\lambda\tau},
	\end{align*}
	where we used the fundamental inequality $(1-x)^{-1}\le 1+2x$ for $x\in[0,\frac{1}{2}]$.

	For any $x\in\mbb R^d$, it holds $\LL x,\Delta W_n\RR\sim \mcal N(0,\tau\|x\|^2)$, which gives that for any $\kappa>0$,
	\begin{align}
		\mbb E[e^{\kappa\LL P_n,\Delta W_n\RR}|\mcal F_{t_n}]=\big(\mbb E[e^{\LL \kappa x,\Delta W_n\RR}]\big)\big|_{x=P_n}=e^{\frac12\kappa^2\tau\|P_n\|^2}. \label{sec4eq2}
	\end{align}
	Since $\Delta W_n$ is independent of $\mathscr{F}_{t_n}$, applying the conditional H\"older inequality and \eqref{sec4eq3}--\eqref{sec4eq2}, we obtain that for any $\tau\in(0,v^{-1}]$ and $\lambda>0$ with $\lambda\tau\le 1/(8C_1)$,
	\begin{align} \label{sec4eq4}
		&\E[e^{\lambda\mathcal{V}_{n+1}^\tau}|\mathscr{F}_{t_n}]\le e^{\lambda (1+v_0\tau)^{-1} \mathcal{V}_{n}^\tau+C_1\lambda\tau}\E[e^{C_1\lambda\|\Delta W_n\|^2+\frac12(1+v_0\tau)^{-1}\lambda\langle P_n, \Delta W_n\rangle}|\mathscr{F}_{t_n}]\\
		&\le e^{\lambda (1+v_0\tau)^{-1} \mathcal{V}_{n}^\tau+C_1\lambda\tau}\left(\E[e^{2C_1\lambda\|\Delta W_n\|^2}]\right)^{\frac12}\left(\E[e^{\lambda(1+v_0\tau)^{-1}\langle P_n, \Delta W_n\rangle}|\mathscr{F}_{t_n}]\right)^{\frac12}\notag\\
		&\le e^{\lambda (1+v_0\tau)^{-1} \mathcal{V}_{n}^\tau+\frac14\lambda^2(1+v_0\tau)^{-2}\tau\|P_n\|^2+C_2\lambda\tau},\notag
	\end{align}
	where $C_2:=C_1(1+2d)>0$ is independent of $\tau$.	For any $\tau\in(0,v^{-1}]$ and $\lambda\in(0,\frac{v_0}{4}]$, it follows from
	$(1+v_0\tau)^{-1}\le 1-\frac{v_0}{2}\tau$ and $\mcal V_n^\tau\ge \frac{1}{4}\|P_n\|^2$ that
	\begin{align}\label{sec4eq5}
		&\lambda (1+v_0\tau)^{-1} \mathcal{V}_{n}^\tau+\frac14\lambda^2(1+v_0\tau)^{-2}\tau\|P_n\|^2+C_2\lambda\tau\\
		&\le \lambda(1-\frac{v_0}{4}\tau)\mathcal{V}_{n}^\tau+\frac{\lambda\tau}{4}\big(-v_0 \mcal V_n^\tau+\lambda(1+v_0\tau)^{-2}\|P_n\|^2\big)+C_2\lambda\tau \notag\\
		&\le\lambda(1-\frac{v_0}{4}\tau)\mathcal{V}_{n}^\tau+C_2\lambda\tau.\notag 
	\end{align}
	Combining \eqref{sec4eq4} and \eqref{sec4eq5}, it holds for any $\tau\in(0,v^{-1}]$ and $\lambda\in(0,\frac{v_0}{4}]$ with $\lambda\tau\le 1/(8C_1)$ that
	$\E[e^{\lambda\mathcal{V}_{n+1}^\tau}|\mathscr{F}_{t_n}]\le e^{\lambda(1-\frac{v_0}{4}\tau)\mathcal{V}_{n}^\tau+C_2\lambda\tau},~n=0,1,2,\ldots$
	By iteration, we conclude that for any $\tau\in(0,v^{-1}]$ and $\lambda\in(0,\frac{v_0}{4}]$ with $\lambda\tau\le 1/(8C_1)$,
	\begin{align*}
		\mbb E\big[e^{\lambda\mathcal{V}_{n+1}^\tau}|\mathscr{F}_{t_{0}}\big]=&\E[\cdots\E[\E[e^{\lambda\mathcal{V}_{n+1}^\tau}|\mathscr{F}_{t_n}]|\mathscr{F}_{t_{n-1}}]\cdots|\mathscr{F}_{t_{0}}]\\
		&\le \E[\cdots\E[\E[e^{\lambda(1-\frac{v_0}{4}\tau)^2\mathcal{V}^\tau_{n-1}+C_2\lambda\tau(1-\frac{v_0}{4}\tau)+C_2\lambda\tau}|\mathscr{F}_{t_{n-2}}]]|\mathscr{F}_{t_{0}}]\\
		&\le e^{\lambda(1-\frac{v_0}{4}\tau)^{n+1}\mathcal{V}^\tau_{0}+\sum_{i=0}^n C_2\lambda\tau(1-\frac{v_0}{4}\tau)^i}\le e^{\lambda\mathcal{V}^\tau_0+4C_2\lambda/v_0}.
	\end{align*}
	Finally, the conclusion comes from letting $\lambda=\kappa_2$ and taking the expectation in the above inequality.
\end{proof}

\begin{cor}\label{cor:Pnunif.}
	Let Assumption \ref{asp:V1} hold. Then for any $\tau\in(0,\tau_1]$ and $p\ge1$,
	$$\sup_{n\in \mathbb{N}}\E\left[\|P_n\|^{2p}+\|Q_n\|^{2p}+r_n^{2p}\right]\le K(P_0,Q_0,r_0,p).$$
\end{cor}

%
	
	\subsection{Uniform-in-time error estimate between $r_n$ and $\sqrt{U_\alpha(Q_n)}$} \label{Sec:3.2}

	Another ingredient in the weak convergence analysis of the DSAV method \eqref{explicitmethod}-\eqref{eq:SAVN} is the uniform-in-time error order between the numerical SAV $r_n$ and the square root of the numerical standardized potential $U_\alpha(Q_n)$, which is stated as follows.
	\begin{lem}\label{lem:rn-sqrt}
		Let Assumption \ref{asp:V1} hold and $r_0=\sqrt{U_\alpha(Q_{0})}$. Then there is a constant $C_3>0$ independent of $\tau$ such that
		$\sup_{n\in \mbb N}\E\big[|r_{n}-\sqrt{U_\alpha(Q_{n})}|^2\big]\le C_3\tau^{2}.
		$
	\end{lem}		
	\begin{proof}
		Denote $\xi_\alpha(q):=\frac{1}{2\sqrt{U_\alpha(q)}}(\nabla V(q)-\alpha q)$.
		Using \eqref{Ualpha}, \eqref{qVq}, and the definition of $\zeta_\alpha$,	we have the following relation
		\begin{equation}\label{eq:identity}
			\langle \xi_\alpha(q),-\frac{v}{2}q\rangle=-\frac{v}{4}\beta_0\sqrt{U_\alpha(q)}-v\zeta_\alpha(q).
		\end{equation}

		It follows from the mean value theorem, \eqref{eq:SAV1Q}, and \eqref{eq:identity} that
		\begin{align}\label{sec4eq7}
			& \sqrt{U_\alpha(Q_{n+1})}-\sqrt{U_\alpha(Q_{n})}
			=\left\langle \int_0^1\xi_\alpha(\theta Q_{n+1}+(1-\theta)Q_n)\ud\theta,\tau P_{n+1}-\frac{v}{2}\tau Q_{n+1}\right\rangle \\
			&  =\left\langle \int_0^1\xi_\alpha(\theta Q_{n+1}+(1-\theta)Q_n)\ud\theta,\tau P_{n+1}\right\rangle-\tau\frac{v}{4}\beta_0\sqrt{U_\alpha(Q_{n+1})}-v\tau\zeta_\alpha(Q_{n+1}) \notag\\
			&\quad+\left\langle \int_0^1\xi_\alpha(\theta Q_{n+1}+(1-\theta)Q_n)- \xi_\alpha(Q_{n+1})\ud\theta,-\frac{v}{2}\tau Q_{n+1}\right\rangle. \notag
		\end{align}
		From \eqref{eq:SAV1R} we obtain that
		\begin{align}
			&r_{n+1}=r_{n}+\tau\langle\xi_\alpha(Q_n),P_{n+1}\rangle-\tau v\zeta_\alpha(Q_n)-\tau\frac{v}{4}\beta_0r_{n+1}. \label{sec4eq6}
		\end{align}
		Combining \eqref{sec4eq6} and \eqref{sec4eq7} gives
		\begin{equation}\label{sec4eq10}
			r_{n+1}-\sqrt{U_\alpha(Q_{n+1})}- r_{n}+\sqrt{U_\alpha(Q_{n})}= -\tau\frac{v}{4}\beta_0\left(r_{n+1}-\sqrt{U_\alpha(Q_{n+1})}\right)+J_n,
		\end{equation}
		where
		\begin{align*}
			J_n:=&\;\tau\Big\langle \int_0^1\big(\xi_\alpha(Q_n)-\xi_\alpha(\theta Q_{n+1}+(1-\theta)Q_n)\big)\ud\theta, P_{n+1}\Big\rangle+v\tau(\zeta_\alpha(Q_{n+1})-\zeta_\alpha(Q_n))\\
			&\;+\frac{v}{2}\tau\Big\langle \int_0^1\xi_\alpha(\theta Q_{n+1}+(1-\theta)Q_n)- \xi_\alpha(Q_{n+1})\ud\theta, Q_{n+1}\Big\rangle.
		\end{align*} 	
		Multiplying $\big(r_{n+1}-\sqrt{U_\alpha(Q_{n+1})}\big)$ on both sides of \eqref{sec4eq10} gives
		\begin{align*}
			& \frac{1}{2}\left[|r_{n+1}-\sqrt{U_\alpha(Q_{n+1})}|^2-|r_{n}-\sqrt{U_\alpha(Q_{n})}|^2\right]\\
			&\le-\tau\frac{v}{4}\beta_0|r_{n+1}-\sqrt{U_\alpha(Q_{n+1})}|^2+J_n(r_{n+1}-\sqrt{U_\alpha(Q_{n+1})})\\
			&\le-\tau\frac{v}{8}\beta_0|r_{n+1}-\sqrt{U_\alpha(Q_{n+1})}|^2+K\tau^{-1}J_n^2.
		\end{align*}
		By \eqref{eq:SAV1Q} and Corollary \ref{cor:Pnunif.}, it holds that
		$\mbb E\|Q_{n+1}-Q_n\|^p\le K\tau ^p$ for any $p\ge 1$.
		Consequently, from 	Corollary \ref{cor:Pnunif.}, the mean value theorem, and $\xi_\alpha,\zeta_\alpha\in C_{\rm p}^{\infty}(\mbb R^d;\mbb R)$ we deduce that for any $p\ge 1$ and $\theta\in[0,1]$,
		\begin{align*}
			\mbb E|\zeta_\alpha(Q_{n+1})-\zeta_\alpha(Q_{n})|^p+\mbb E\|\xi_\alpha(\theta Q_{n+1}+(1-\theta)Q_n)-\xi_\alpha(Q_i)\|^p\le K\tau ^p, 
		\end{align*}
		where $i\in\{n,n+1\}$.
		Accordingly, we have
		$\mathbb{E}[J_n^2]\le C\tau^4$,
		which implies 
		\begin{align*}
			(1+\frac{v\beta_0}{4}\tau)\E[|r_{n+1}-\sqrt{U_\alpha(Q_{n+1})}|^2] \le \E[|r_{n}-\sqrt{U_\alpha(Q_{n})}|^2] +K\tau^3.
		\end{align*}
		Finally, the conclusion comes by iterating the above inequality and 	$r_0=\sqrt{U_\alpha(Q_0)}$.
	\end{proof}

	\subsection{Regularity of  the Kolmogorov equation}
	Denote $u^\psi(t,p,q):=\\$$\E[\psi(\mathfrak P(t,p,q),\mathfrak Q(t,p,q))]$, $t\ge0$, $(p,q)\in\mbb R^{2d}$, and
	$$B(p,q):=\left(
	\begin{array}{c}
		-\nabla V(q)-\frac{v}{2}p+\frac{v^2}{4} q \\p-\frac{v}{2}q 
	\end{array}
	\right),~G(p,q):=\left(
	\begin{array}{c}
		I \\0
	\end{array}
	\right).$$
	Then	 $u^\psi$ 
	solves the following Kolmogorov equation associated with \eqref{eq:Langevin}:
	\begin{align}\label{eq:Kol}
		\partial_t u^\psi(t,p,q)=\langle \nabla u^\psi(t,p,q),B(p,q)\rangle +\frac12\mathrm{tr}(GG^\top\nabla ^2u^\psi(t,p,q)),\quad t>0
	\end{align}
	with the initial value $u^\psi(0,p,q)=\psi(p,q)$ for $p,q\in\R^d$.
	
	Establishing uniform-in-time weak convergence of the DSAV method via the Kolmogorov equation relies crucially on the exponential decay properties of the solution's partial derivatives, as presented in the following lemma. According to the definition of $D^kf(\xi_1,\ldots,\xi_k)$ for $f\in C^k(\mbb R^d;\mbb R)$, $D^kf$ can be viewed as a tensor, and we use the notation $\|\cdot\|_{\otimes}$ to denote the norm of a tensor. 
	\begin{lem}\label{lem:exp}
		Suppose that Assumption \ref{asp:V0} and Condition \ref{condition1} hold.		
		Assume that $\psi\in C^\infty_{\rm p}(\mbb R^{2d};\mbb R)$, and denote $u^\psi(t, p,q):=\E[\psi(\mathfrak{P}(t,p,q),\mathfrak Q(t,p,q))]$. Then for any integer $m$, there exist an integer $\eta(m)$ and constants $L_3>0$ and $\gamma>0$ such that 
		\begin{align*}
			\|D^m u^\psi(t,p,q)\|_{\otimes}\le L_3(1+\| p\|^{\eta(m)}+\|q\|^{\eta(m)})e^{-\gamma t}\quad \forall~t>0,\;p,\,q\in\R^d.
		\end{align*}
		Here, $D^m u^\psi(t,p,q)$ denotes the $m$th derivative of $u^\psi$ with respect to $(p,q)$. 		
	\end{lem}
	\begin{proof}
		Introduce the function $\phi(p,q):=\psi(p+\frac{v}{2}q,q)$, $p,q\in\R^d$. Clearly, $\phi\in C^\infty_{\rm p}(\mbb R^{2d};\mbb R)$. Furthermore, define $\bar u^\phi(t, p,q):=\E[\phi(\bar{\mathfrak{P}}(t, p,q),\bar{\mathfrak Q}(t, p,q))]$ for $t\ge 0$ and $(p,q)\in\mbb R^{2d}$.
		Noting that $\mathfrak P(t,p,q)=\bar{\mathfrak{P}}(t,p-\frac{v}{2}q,q)+\frac{v}{2}\bar{\mathfrak Q}(t, p-\frac{v}{2}q,q)$ and $\mathfrak Q(t,p,q)=\bar{\mathfrak Q}(t,p-\frac{v}{2}q,q)$,
		we have 
		\begin{align}\label{eq:upsi}
			u^{\psi}(t,p,q)=\bar u^\phi(p-\frac{v}{2}q,q)\quad\forall~(p,q)\in\mbb R^{2d}.
		\end{align} 
		
		We note that \eqref{eq:Langevinorg} can be written as \cite[Eq.\ (1.1)]{T02} with $\mbb H(x,y)=\frac{1}{2}|y|^2+V(x)$ and $\mbb F(x,y)\equiv v$. It was shown in \cite[Lemma 3.3]{MSH02} that $\mcal L(\mbb H+ R)\le -a_1(\mbb H+R)+d_1$ for two constants $a_1,d_1>0$, where $\mcal L$ is the infinitesimal generator of \eqref{eq:Langevinorg} and $R(x,y)=\frac{v}{2}\LL x,y\RR+\frac{v^2}{4}\|x\|^2+1$. This, combined with Assumption \ref{asp:V0} and Condition \ref{condition1}, ensures that \cite[Hypothesis 1.1]{T02} is satisfied by our setting. 	According to \cite[Theorem 3.1]{T02},
		for any integer $m$, there exist a positive integer $\eta(m)$ and $K,\gamma>0$ such that 
		\begin{align*}
			\|D^m \bar u^\phi(t, \bar p,q)\|_{\otimes}\le K(1+\|\bar p\|^{\eta(m)}+\|q\|^{\eta(m)})e^{-\gamma t},\quad \forall~t>0,\; p,\,q\in\R^d.
		\end{align*}
		The proof is finished by the above relation and \eqref{eq:upsi}.
	\end{proof}

	With these preparations in place, we are now in a position to prove the uniform-in-time weak convergence of the DSAV method \eqref{explicitmethod}-\eqref{eq:SAVN}.
	\begin{lem}\label{weakorder}
		Let Assumptions \ref{asp:V1} and \ref{asp:V0} hold. Consider the numerical solution $\{(P_n,Q_n)\}_{n\ge0}$ generated by \eqref{eq:SAVN} with $r_0=\sqrt{U_\alpha(Q_0)}$. Then for any $(P_0,Q_0)\in\mbb R^{2d}$ and $\psi\in C^\infty_{\rm p}(\mbb R^{2d};\mbb R)$,  there is a constant $C_4>0$ independent of $\tau$ such that
		$$\sup_{N\in\mbb N}\big|\E[\psi(P_N,Q_N)]-\E[\psi(\mathfrak P(t_N,P_0,Q_0),\mathfrak Q(t_N,P_0,Q_0))]\big|\le C_4\tau.$$
	\end{lem}	
	Based on Corollary \ref{cor:Pnunif.} and Lemmas \ref{lem:rn-sqrt}–\ref{lem:exp}, Lemma \ref{weakorder} can be proved via the standard weak error decomposition based on the Kolmogorov equation. For the sake of completeness, we provide the detailed proof in the appendix.
	
	\emph{Proof of Theorem \ref{theo:weakorder}.} For any given $\varphi\in C^\infty_{\rm p}(\mbb R^{2d};\mbb R)$, denote $\psi(p,q)=\varphi(p-\frac{v}{2}q,q)$, $(p,q)\in\mbb R^{2d}$. Since 
	$
	\bar{\mathfrak P}(t,p,q)+\frac{v}{2}\bar{\mathfrak Q}(t,p,q)=\mathfrak P(t,p+\frac{v}{2}q,q)$ and $\bar{\mathfrak Q}(t,p,q)=\mathfrak Q(t,p+\frac{v}{2}q,q),
	$
	we have 
	$\E[\varphi(\bar{\mathfrak P}(t_n,p,q),\bar{\mathfrak Q}(t_n,p,q))]
	=\E\big[\psi\big(\mathfrak P(t_n,p+\frac{v}{2}q,q),\mathfrak Q(t_n,p+\frac{v}{2}q,q)\big)\big]$.
	Furthermore, $\E[\varphi(\bar{\mathfrak{P}}_n,\bar{\mathfrak{Q}}_n)]=\E[\varphi(P_n-\frac{v}{2}Q_n,Q_n)]=\E[\psi(P_n,Q_n)]$. Combining the above formulas and Lemma \ref{weakorder}, we complete the proof. 	 \hfill $\square$

	We remark that the weak convergence can also be used in analyzing the complexity of the Monte Carlo and multilevel Monte Carlo methods (cf.\ \cite{GM08,DHSZ23}).

	\section{Proof of Theorem   \ref{thm:density}} \label{Sec:aapromeasure} In this section, we present the proof of Theorem   \ref{thm:density}.
	
	\emph{Proof of Theorem \ref{thm:density}.}
	Taking \eqref{eq:Qn+1}--\eqref{eq:rn+1-old} into account,
	there is a smooth map $g:\R^{2d}\times \R\times\R^{2d}\to \R^{2d}$ such that $(P_{n+2}^\top,Q_{n+2}^\top)^\top=g(P_n,Q_n,r_n,\Delta W_n,\Delta W_{n+1})=:g_{P_n,Q_n,r_n}(\Delta W_n,\Delta W_{n+1})$. 
	
	\textbf{Step 1.}
	In this step, we show that for any fixed $p,q,r$, the map $g_{p,q,r}:\R^{2d}\to\R^{2d}$ is invertible and its Jacobi matrix is invertible, namely,
	$\det\left(\frac{\partial g_{p,q,r}(x,y)}{\partial (x,y)}\right)\ne 0.$
	It suffices to show that for any fixed $P_0,Q_0,r_0$, the Wiener increments
	$\Delta W_0,\Delta W_1$ can be uniquely determined by $(P_2,Q_2,P_0,Q_0,r_0)$, and
	\begin{equation}\label{eq:Dg}
		\det\left(\frac{\partial g_{P_0,Q_0,r_0}(\Delta W_0,\Delta W_1)}{\partial (\Delta W_0,\Delta W_1)}\right)\ne 0.
	\end{equation}
	By \eqref{eq:Qn+1}, we can determine $Q_1$ and $P_1$ via
	$Q_1=(1+\frac{v}{2}\tau)Q_{2}-\tau P_{2}$ and $ P_{1}=\tau^{-1}((1+\frac{v}{2}\tau)Q_{1}-Q_0)$, respectively.
	Utilizing \eqref{eq:Pn+1new}, we have
	\begin{align*}
		\Delta W_0=\left(-2\tau^2G_0(Q_0)^{-1} G_2(Q_0) G_2(Q_0)^\top+ f_{v,\tau}I\right)^{-1} \left(P_1-P_0-\tau\mathcal{P}(P_0,Q_0,r_0)\right),
	\end{align*}
	which along with \eqref{eq:rn+1-old} determines $r_1$, i.e.,
	\begin{align}
		r_{1}=G_0(Q_0)^{-1} \tau G_2(Q_0)^\top \Delta W_0+G_0(Q_0)^{-1}r_0+G_0(Q_0)^{-1}G_1(Q_0,P_0).
	\end{align}
	Finally, since $r_1, P_1,Q_1$ are uniquely determined by $P_2$ and $Q_2$, $\Delta W_1$ can be uniquely determined by $P_2$ and $Q_2$ through \eqref{eq:Pn+1new}. In fact, the above procedure implies that $(\Delta W_0,\Delta W_1):=g^{-1}_{P_0,Q_0,r_0}(P_2,Q_2)$ is smooth with respect to $(P_0,Q_0,r_0,P_2,Q_2)$.
	Next, we show that the Jacobi matrix $Dg_{P_0,Q_0,r_0}$ of $g_{P_0,Q_0,r_0}$ is invertible. 
	It follows from \eqref{eq:Qn+1} that                   	\begin{align*}
		&Dg_{P_0,Q_0,r_0}( \Delta W_0, \Delta W_1):=\left[\begin{array}{cc}\frac{\partial P_2}{\partial \Delta W_0} & \frac{\partial P_2}{\partial \Delta W_1} \\\frac{\partial Q_2}{\partial \Delta W_0} & \frac{\partial Q_2}{\partial \Delta W_1}\end{array}\right]\\
		&=\left[\begin{array}{cc}\frac{\partial P_2}{\partial \Delta W_0} & \frac{\partial P_2}{\partial \Delta W_1} \\\!\!\!(1+\frac{v}{2}\tau)^{-1}\frac{\partial Q_1}{\partial \Delta W_0}+\tau(1+\frac{v}{2}\tau)^{-1}\frac{\partial P_{2}}{\partial \Delta W_0} & (1+\frac{v}{2}\tau)^{-1}\frac{\partial Q_1}{\partial \Delta W_1}+\tau(1+\frac{v}{2}\tau)^{-1}\frac{\partial P_{2}}{\partial \Delta W_1}\!\!\! \end{array}\right].
	\end{align*}
	Noting that $\frac{\partial Q_0}{\partial \Delta W_0}=\frac{\partial Q_1}{\partial \Delta W_1}=0$ and using \eqref{eq:Qn+1}, one has		
	\begin{align}\label{eq:Dg12}\notag
		\det(Dg_{P_0,Q_0,r_0}( \Delta W_0, \Delta W_1))&=\det\left(\left[\begin{array}{cc}\frac{\partial P_2}{\partial \Delta W_0} & \frac{\partial P_2}{\partial \Delta W_1} \\\left(1+\frac{v}{2}\tau\right)^{-1}\frac{\partial Q_1}{\partial \Delta W_0}& 0\end{array}\right]\right)\\\notag
		&=\det\left(\left[\begin{array}{cc}\frac{\partial P_2}{\partial \Delta W_0} & \frac{\partial P_2}{\partial \Delta W_1} \\\tau\left(1+\frac{v}{2}\tau\right)^{-2}\frac{\partial P_{1}}{\partial \Delta W_0}& 0\end{array}\right]\right)\\
		&=(-1)^d\tau^d\left(1+\frac{v}{2}\tau\right)^{-2d}\det\left( \frac{\partial P_2}{\partial \Delta W_1}\right)\det\left( \frac{\partial P_1}{\partial \Delta W_0}\right).
	\end{align}
	According to \eqref{eq:Pn+1new} and the positivity of $G_0(Q_n)$, we have
	\begin{align*}
		\det\left( \frac{\partial P_{n+1}}{\partial \Delta W_n}\right)=\det\Big(-2\tau^2G_0(Q_n)^{-1} G_2(Q_n) G_2(Q_n)^\top+ f_{v,\tau}I\Big) .
	\end{align*}
	By \eqref{eq:G0} and \eqref{eq:G2}, it holds that
	\begin{align*}
		G_3^n:=&\;-2\tau^2G_0(Q_n)^{-1} G_2(Q_n)G_2(Q_n)^\top+ f_{v,\tau}I\\
		=&\;-\left(1+\frac{v}{4}\tau\beta_0+\frac{\tau^2f_{v,\tau}}{2U_\alpha(Q_n)} \|\nabla V(Q_n)-\alpha Q_n\|^2\right)^{-1}\\
		&\; \times\frac{ \tau^2f_{v,\tau}^2}{2U_\alpha(Q_n)}(\nabla V(Q_n)-\alpha Q_n)(\nabla V(Q_n)-\alpha Q_n)^\top+f_{v,\tau}I.
	\end{align*}
	Then the matrix $G_3^n$ has one eigenvalue	
	\begin{equation*}
		f_{v,\tau}\left(1+\frac{v}{4}\tau\beta_0+\frac{\tau^2f_{v,\tau}}{2U_\alpha(Q_n)} \|\nabla V(Q_n)-\alpha Q_n\|^2\right)^{-1}\times
		\left(1+\frac{v}{4}\tau\beta_0\right),
	\end{equation*}
	and the other $(d-1)$ eigenvalues are all $f_{v,\tau}$.
	Hence $\det\left( \frac{\partial P_{n+1}}{\partial \Delta W_n}\right)\neq 0$ for $n\in\mbb N$,
	which, in combination with \eqref{eq:Dg12}, completes the proof of \eqref{eq:Dg}.
	
	\textbf{Step 2.} In this step, we show the existence and positivity of the density of $(P_{n},Q_{n})$ for $n\ge2$.
	For any $f\in\mscr B_b(\R^{2d};\R)$, by the change of variables, we have
	\begin{align*}
		&\mbb E (f(P_{n+2},Q_{n+2}))=		\E\big[\E(f(g(P_n,Q_n,r_n,\Delta W_n,\Delta W_{n+1}))|\mathcal{F}_{t_n})\big]\\
		&\!=\E\int_{\R^{2d}}f(g_{P_n,Q_n,r_n}(x,y))\rho_{(\Delta W_{n},\Delta W_{n+1})}(x,y)\ud x\ud y\\
		&\!=\E\int_{\R^{2d}}\!f(p,q)\rho_{(\Delta W_{n},\Delta W_{n+1})}(g^{-1}_{P_n,Q_n,r_n}(p,q))\frac{1}{|\det(Dg_{P_n,Q_n,r_n}(g_{P_n,Q_n,r_n}^{-1}(p,q)))|}\ud p\ud q,
	\end{align*}
	where $\rho_{(\Delta W_n,\Delta W_{n+1})}$ denotes the density of $(\Delta W_n,\Delta W_{n+1})$. 
	Hence, $(P_{n+2},Q_{n+2})$ admits the density function $\rho_{(P_{n+2},Q_{n+2})}:\R^{2d}\to \R$, $(p,q)\mapsto 	\rho_{(P_{n+2},Q_{n+2})}(p,q)$ with
	\begin{align*}
		&\rho_{(P_{n+2},Q_{n+2})}(p,q)\\&=\E[\rho_{(\Delta W_{n},\Delta W_{n+1})}(g^{-1}_{P_n,Q_n,r_n}(p,q))|\det(Dg_{P_n,Q_n,r_n}(g_{P_n,Q_n,r_n}^{-1}(p,q)))|^{-1}].
	\end{align*}
	For any nonempty open set $U\subseteq\R^{2d}$ and $(P^\top,Q^\top,r)^\top\in\R^{2d+1}$,
	since $g_{P,Q,r}:\R^{2d}\to \R^{2d}$ is a continuous bijection, $g_{P,Q,r}^{-1}(U)$ is a nonempty open set. This yields that 
	\begin{small}	\begin{equation*}
			\mathbb{P}\{(P_{n+2}^\top,Q_{n+2}^\top)^\top\in U\mid P_n=P,Q_n=Q,r_n=r\}=\mathbb{P}\{(\Delta W_n^\top,\Delta W_{n+1}^\top)^\top\in g_{P,Q,r}^{-1}(U)\}>0.
	\end{equation*}	\end{small}	
	Consequently, by the independence of $(\Delta W_n^\top,\Delta W_{n+1}^\top)^\top$ and $X_n:=(P_n^\top,Q_n^\top,r_n)^\top$, we have that for any nonempty open set $U\subseteq\R^{2d}$,
	\begin{align*}
		\mathbb{P}\{(P_{n+2}^\top,Q_{n+2}^\top)^\top\in U\}=\int_{\R^{2d+1}}\mathbb{P}\{(\Delta W_n^\top,\Delta W_{n+1}^\top)^\top\in g_{P,Q,r}^{-1}(U)\}\ud \mathbb{P}\circ X_n^{-1}>0,
	\end{align*}
	which, together with the outer regularity of the Lebesgue measure, indicates that $\rho_{(P_{n+2},Q_{n+2})}$ is $Leb(\mbb R^{2d})$-a.e. positive via proof by contradiction.  Thus, the proof of Theorem \ref{thm:density} is complete. \hfill $\square$             
	%
	\begin{rem} \label{rem:num-PDF}
		We remark that the law of $(P_1,Q_1)$ is not absolutely continuous with respect to $Leb(\mbb R^{2d})$, and thus $(P_1,Q_1)$ does not admit a density function. In fact, it follows from \eqref{eq:SAV1Q} that $\mbb P\{(P_1,Q_1)\in S\}=1$, with $S=\{(p,q)\in\mbb R^{2d}\,:\, (1+\frac{v\tau}{2})q-\tau p=Q_0\}$, which means that the law of $(P_1,Q_1)$ is singular with respect to $Leb(\mbb R^{2d})$ due to $Leb(\mbb R^{2d})(S)=0$. 
	\end{rem}                \begin{rem}\label{rem:exact-PDF}
		According to Lemma \ref{equivalent}, for any initial value $(P(0),Q(0),r(0))\in \mathcal{M}=\{(p,q,r)\in\R^{2d+1}: r=\sqrt{U_\alpha(q)}\}$, the solution $(P(t),Q(t),r(t))$ of \eqref{eq:SAV1} lies in $\mathcal{M}$, for $\mbb P$-almost sure $\omega\in\Omega$ and $t\ge 0$. In this case, for any $t\ge 0$, the law of $(P(t),Q(t),r(t))$ does not have a density function. Thus, the vector fields of \eqref{eq:SAV1} do not satisfy the parabolic H\"ormander condition. 
	\end{rem}       	
	
	It follows from Theorem \ref{thm:density} and \eqref{explicitmethod} that for any $n\ge 2$, $(\bar{\mathfrak{P}}_n,\bar{\mathfrak{Q}}_n)$ admits a density function. Hence, we obtain an equivalent form of \eqref{sec5eq1} via density functions.  Recall that $\widetilde{\rho}$ is the density function of $\widetilde{\mu}$ given in Section \ref{Sec:1.1}.
	\begin{cor}\label{cor:appromeasure2}
		Let Assumptions \ref{asp:V1} and \ref{asp:V0} hold and $(p,q)\in\mbb R^{2d}$ be given. Consider the numerical solution $\{(\bar{\mathfrak{P}}_n,\bar{\mathfrak{Q}}_n)\}_{n\ge 0}$ generated by the DSAV method \eqref{explicitmethod}-\eqref{eq:SAVN} with the initial value $(P_0,Q_0,r_0)=(p+\frac{v}{2}q,q,\sqrt{U_\alpha(q)})$. Then for any $n\ge 2$, $(\bar{\mathfrak{P}}_n,\bar{\mathfrak{Q}}_n)$ admits a density function, denoted by $\rho_n$. In addition, for any $\phi\in C^\infty_{\rm p}(\mbb R^{2d};\mbb R)$, 
		\begin{align*}
			\Big|\int_{\R^{2d}}\phi(p,q)\big(\rho_n(p,q)-\widetilde{\rho}(p,q)\big)\,\ud p\ud q\Big|\le C(e^{-K(\phi)t_n}+\tau),\quad n\ge 2.
		\end{align*}
	\end{cor}
	
	\section{Numerical experiments}\label{Sec:num-exp}
	In this section, we perform several numerical experiments to verify our theoretical results and the superiority of the DSAV method \eqref{explicitmethod}-\eqref{eq:SAVN}. 
	In the experiments, we take the potential $V(q)=\frac{1}{6}(q^3-1)^2+\sin q+2$, $q\in\mbb R$ and $v=1$. Then Assumptions \ref{asp:V0} and \ref{asp:V1} are fulfilled according to Remark \ref{rem:example}, and the parameters in Conditions \ref{condition1} and \ref{condition2} can be set as $\alpha=\frac{1}{2}$, $C_V=0$, and $\beta_0=6$. Throughout this section, the initial value is always set as $(\bar{\mathfrak{P}}(0),\bar{\mathfrak{Q}}(0))=(1,0.5)$.

	First, we test the convergence order of the DSAV method.  We set $T=1$ and compute the numerical solution $(\bar{\mathfrak{P}}_{N},\bar{\mathfrak{Q}}_{N})$ for different stepsizes $\tau=2^{-8},2^{-9},\ldots,2^{-12}$ with $N=T/\tau$. The numerical solution with $\tau=2^{-14}$ is viewed as the ``exact" solution. The weak error is characterized by $err_{w}:=\big|\mbb E\big[\phi(\bar{\mathfrak{P}}_{N},\bar{\mathfrak{Q}}_{N})-\phi(\bar{\mathfrak{P}}(T),\bar{\mathfrak{Q}}(T))\big]\big|$ for three different test functions $\phi_1(p,q)=p^2+q^2$, $\phi_2(p,q)=\sin p\sin q$, and $\phi_3(p,q)=p^4q^4$. Figure \ref{F:order} (a) displays the first-order weak convergence of the DSAV method, coinciding with the theoretical result in Theorem \ref{theo:weakorder}. In addition, we also test the strong convergence order of the DSAV method and the result is depicted in Figure \ref{F:order} (b), where the strong error is computed by $err_{s}:=\big(\mbb E\big[|\bar{\mathfrak{P}}_{N}-\bar{\mathfrak{P}}(T)|^2+|\bar{\mathfrak{Q}}_{N}-\bar{\mathfrak{Q}}(T)|^2\big]\big)^{1/2}$. In this experiment, the expectation is realized based on the Monte Carlo method with the sample sizes $10000$ for weak error and $1000$ for strong error.

	Second, we examine the long-time weak error evolution of the DSAV method.  The numerical solution is computed with the stepsize $\tau=10^{-3}$, and the exact one is realized via the numerical one with $\tau=10^{-4}$. Figure \ref{F:longtimeweak} shows the evolution of the weak error $err_w$ in the time range $[0,100]$, for three different test functions. The expectation is approximated based on the Monte Carlo method with a sample size $10000$. It is observed that the error remains at a small level after the long-time computation, consistent with the result of Theorem \ref{theo:weakorder}.

	Next, we conduct numerical experiments to validate Corollary \ref{theo:appromeasure1}, which asserts that $\mathbb{E}\phi(\bar{\mathfrak{P}}_n,\bar{\mathfrak{Q}}_n)$ provides a robust approximation of the ergodic limit $\widetilde{\mu}(\phi)$ for polynomially growing test functions $\phi$ as time $t_n$ increases and stepsize $\tau$ decreases. We set the stepsize to $\tau=10^{-3}$, and the expectation is approximated based on the Monte Carlo method with a sample size $10000$. Figure \ref{F:ergodic} illustrates the evolution of the error, defined as $ERR:=|\mathbb{E}\phi(\bar{\mathfrak{P}}_n,\bar{\mathfrak{Q}}_n)-\widetilde{\mu}(\phi)|$, for the test functions $\phi \in \{\phi_1, \phi_2, \phi_3\}$. It is observed that the DSAV method achieves a highly accurate approximation, with the error $ERR$ decaying rapidly toward zero.

	Finally, we evaluate the long-time moment stability of the DSAV method by monitoring the evolution of the $p$th moment of the numerical solution, $\mathbb{E}[|\bar{\mathfrak{P}}_n|^{p}+|\bar{\mathfrak{Q}}_n|^{p}]$, for $p \in \{2, 6\}$, with the stepsize $\tau=10^{-2}$. We also include the SSAV method proposed in \cite{DJW25} for comparative analysis. The expectation is approximated based on the Monte Carlo method with a sample size $10000$. As illustrated in Figure \ref{F:moment}, the second and sixth moments of the DSAV method remain bounded over time, whereas those of the SSAV method exhibit polynomial growth. This confirms the superior long-time moment stability of the DSAV method.

	
	\begin{figure}[!htb]
		\centering
		\subfigure[Weak order]{\includegraphics[width=0.48\textwidth]{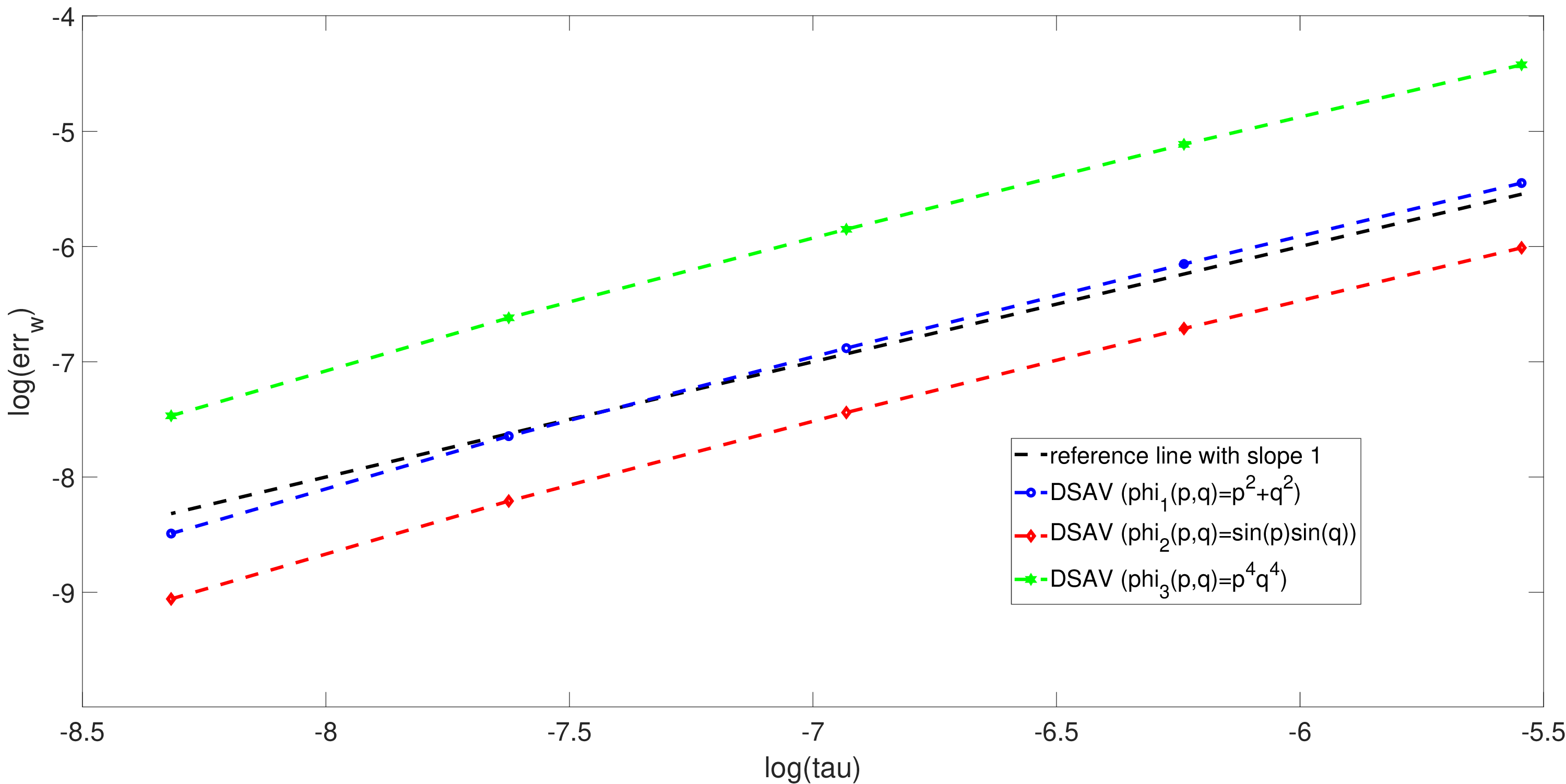}}
		\subfigure[Strong order]{\includegraphics[width=0.48\textwidth]{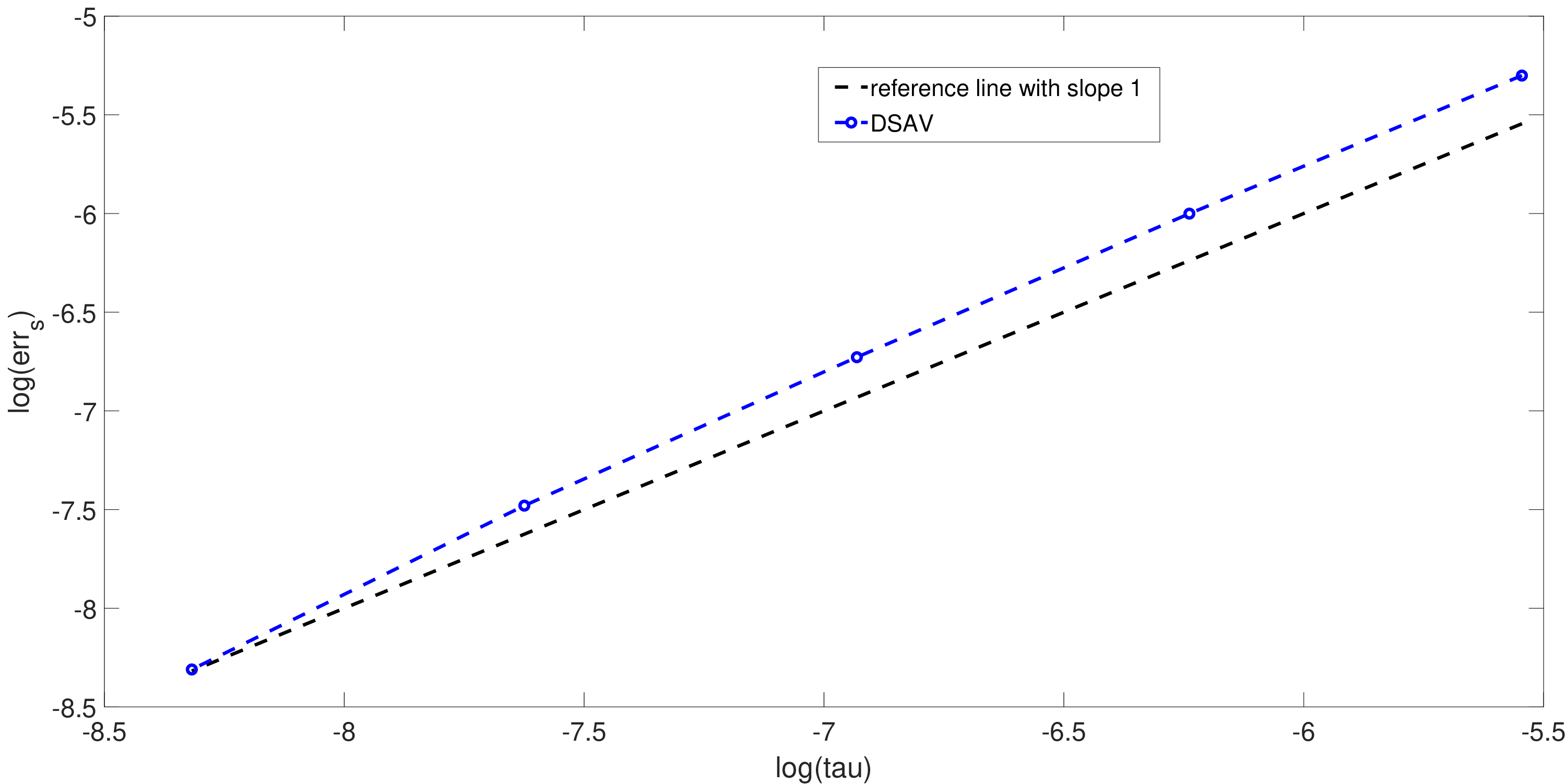}}
		\caption{Convergence order of the numerical method the DSAV method \eqref{explicitmethod}-\eqref{eq:SAVN}.}
		\label{F:order}
	\end{figure}

	\begin{figure}[!htb]
		\centering
		\includegraphics[width=0.96\textwidth]{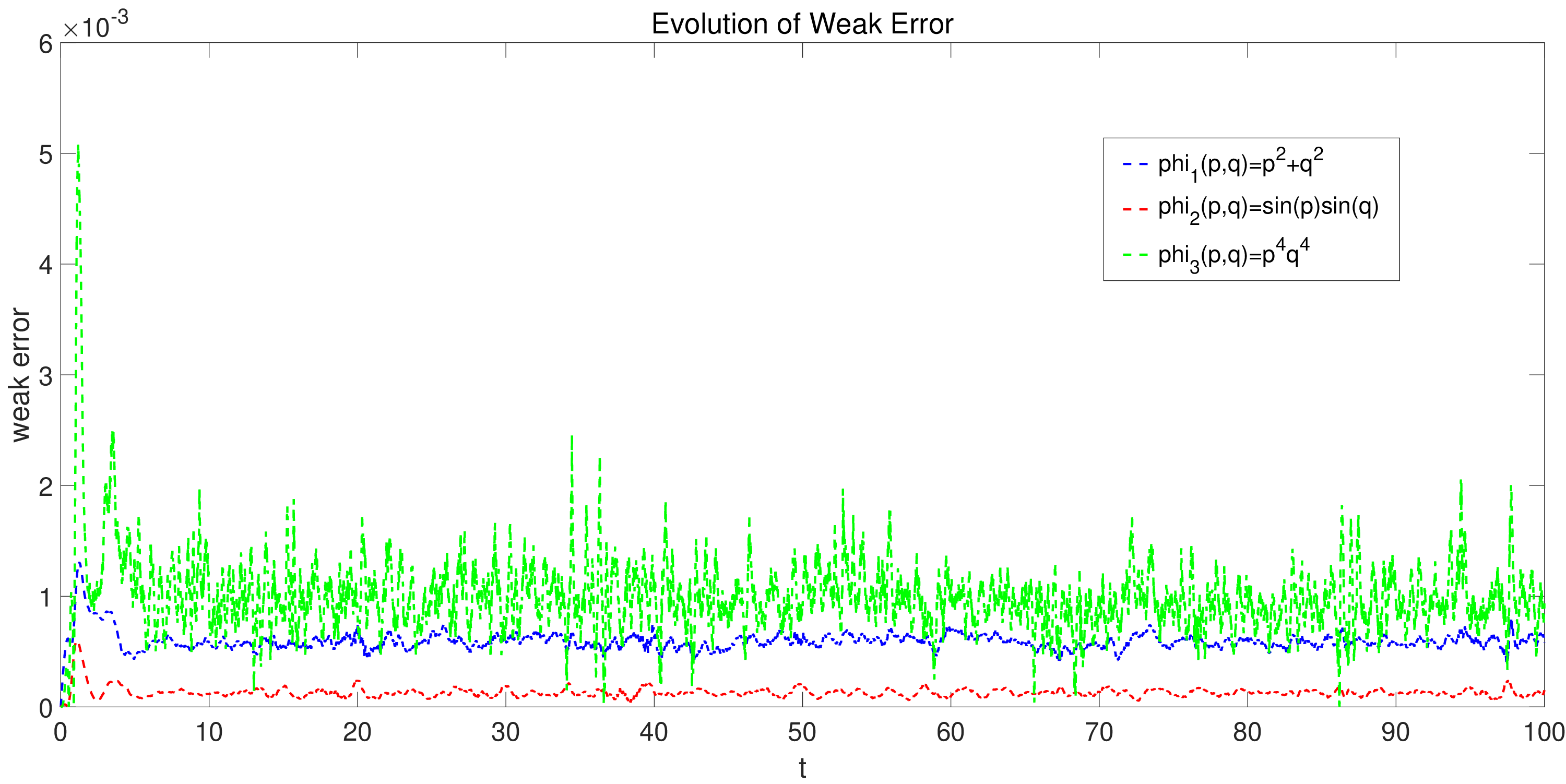}
		\caption{Long-time weak error evolution for different test functions.}
		\label{F:longtimeweak}
	\end{figure}
	
	\begin{figure}[!htb]
		\centering
		\subfigure[$\phi(p,q)=p^2+q^2$]{\includegraphics[width=0.48\textwidth]{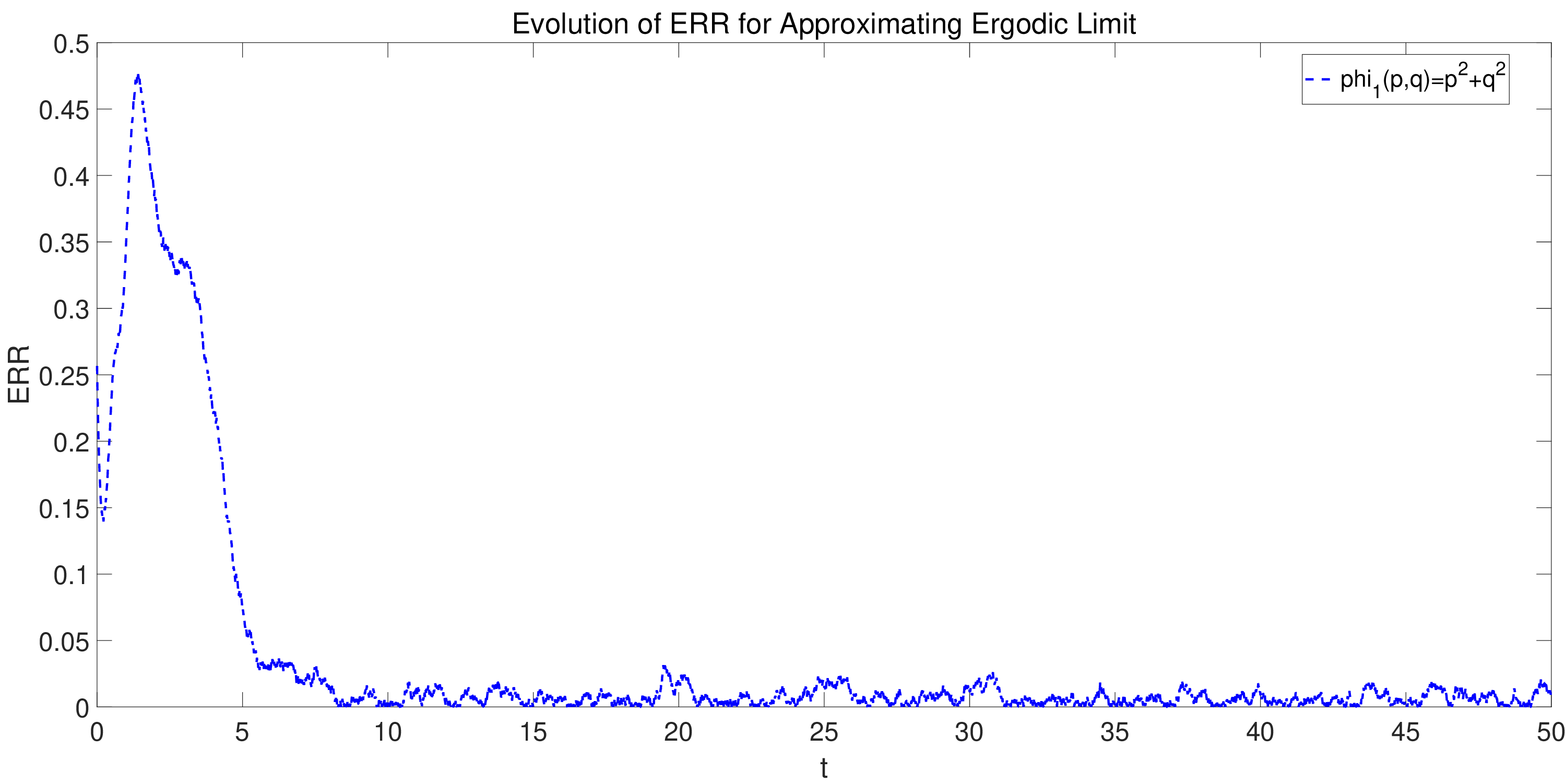}}
		\subfigure[$\phi(p,q)=p^4q^4$]{\includegraphics[width=0.48\textwidth]{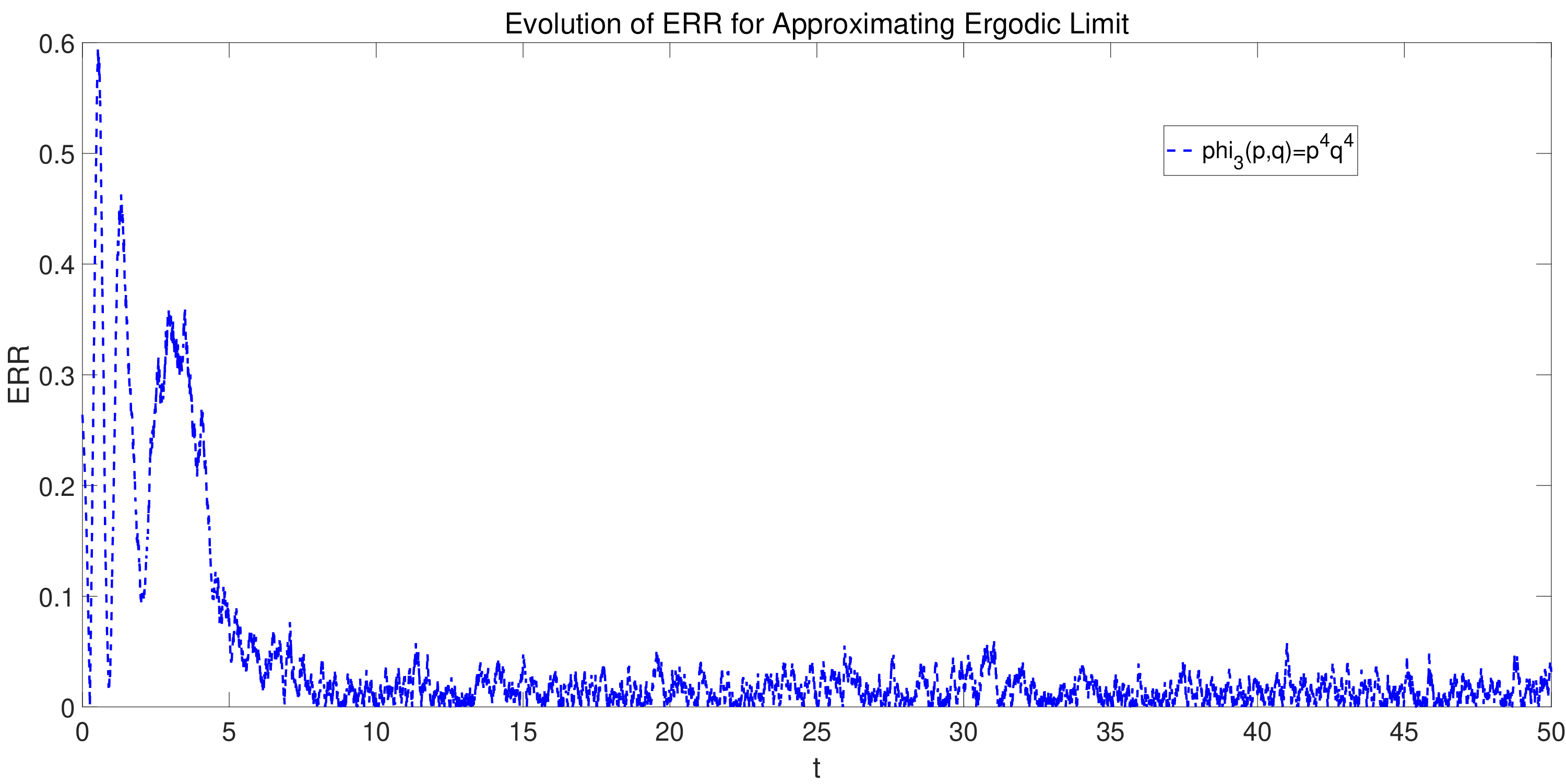}}
		\caption{Evolution of $ERR:=|\mbb E\phi(\bar{\mathfrak{P}}_n,\bar{\mathfrak{Q}}_n)-\widetilde{\mu}(\phi)|$ for different test functions.}
		\label{F:ergodic}
	\end{figure}
	
	\begin{figure}[!htb]
		\centering
		\subfigure[$p=2$]{\includegraphics[width=0.48\textwidth]{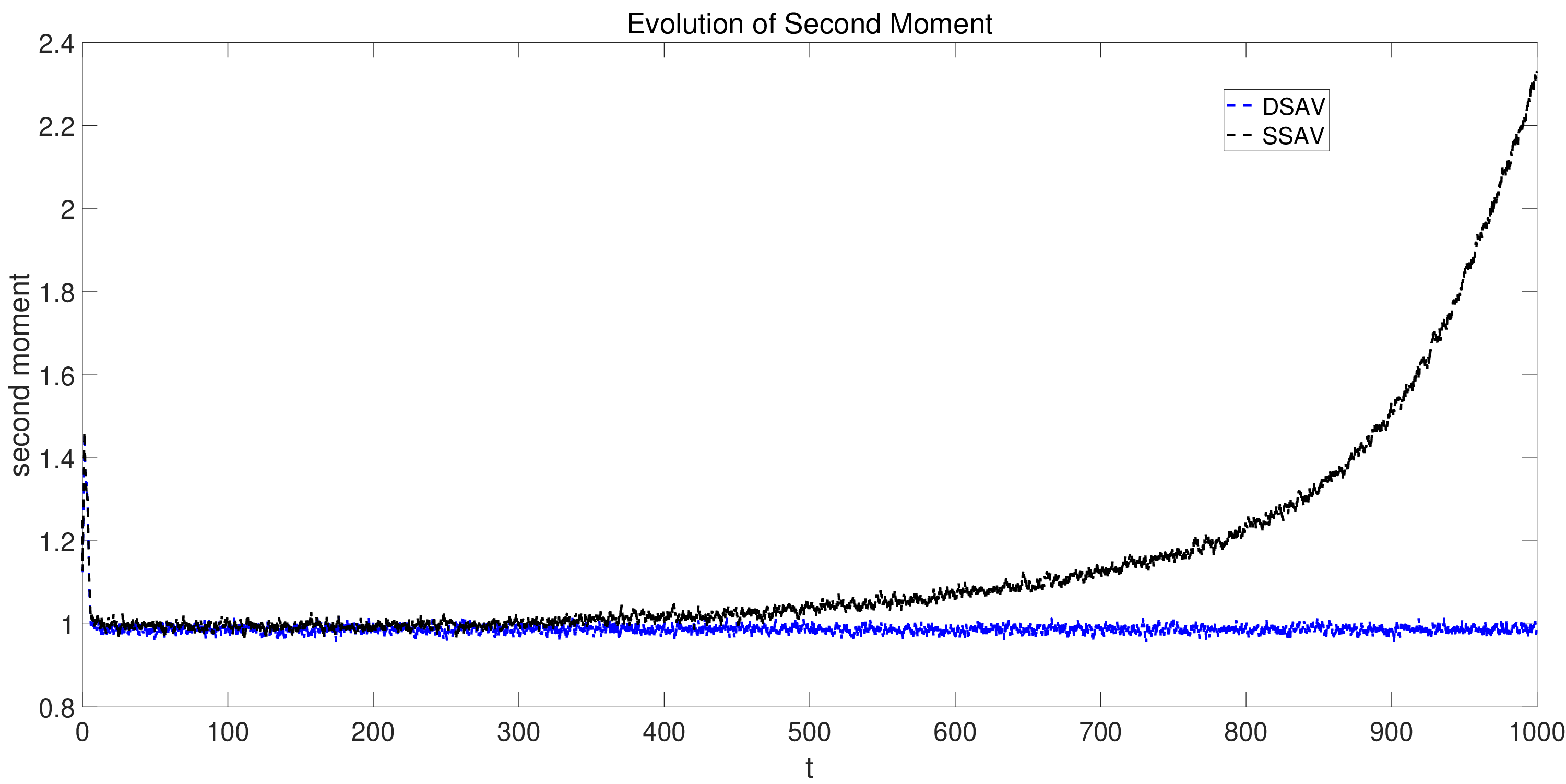}}
		\subfigure[$p=6$]{\includegraphics[width=0.48\textwidth]{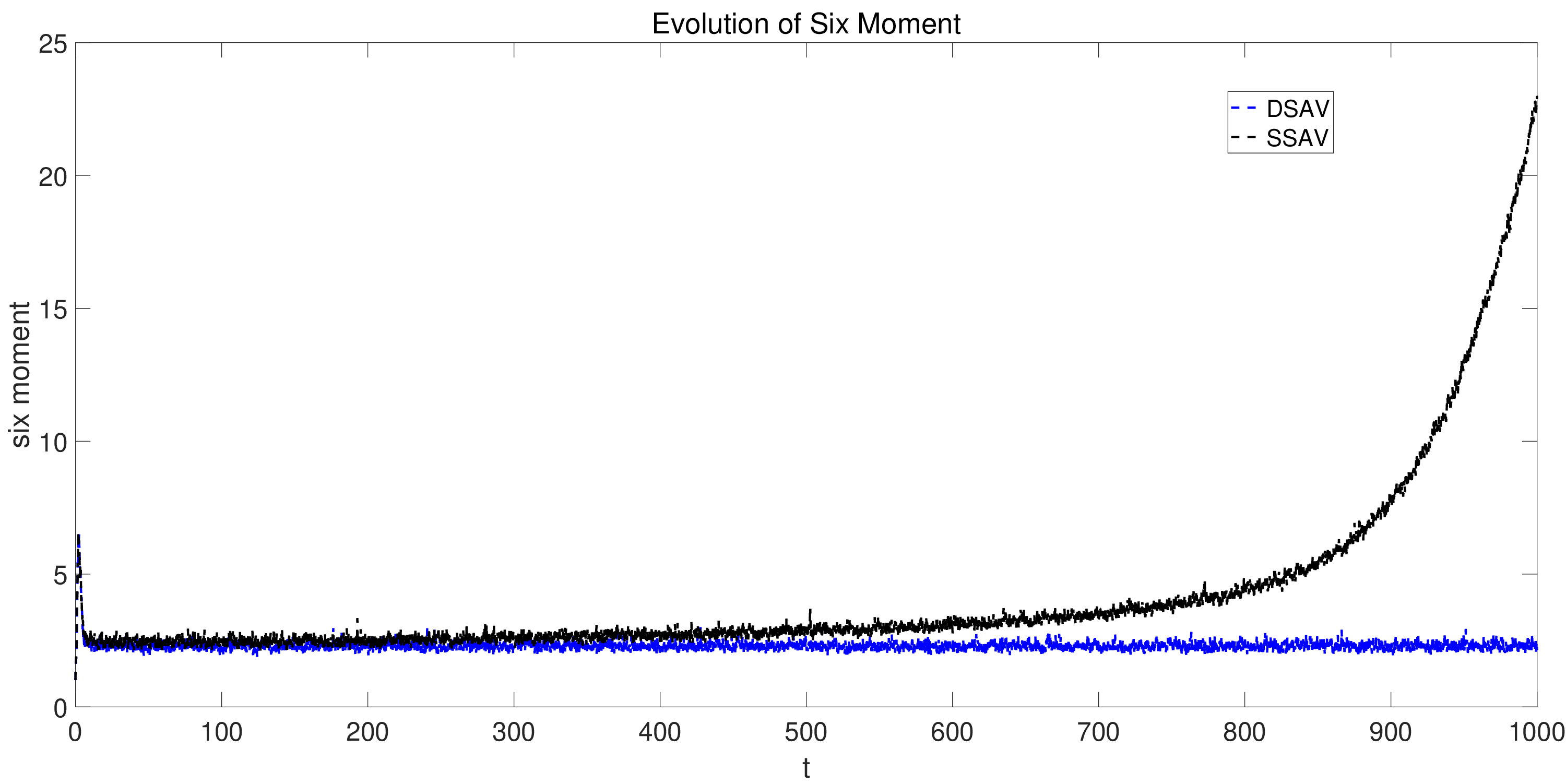}}
		\caption{Eevolution of $p$th moment of numerical solution.}
		\label{F:moment}
	\end{figure}

	\section{Concluding remarks} \label{Sec:conclusion}
	In this work, we construct a new explicit numerical method  for the underdamped Langevin equation \eqref{eq:Langevinorg} with polynomially growing potentials. We give the uniform-in-time weak convergence  and  the convergence speed for approximating the invariant measure of \eqref{eq:Langevinorg}. In addition, the existence and positivity of the density function of the numerical solution are also analyzed.	
	
	We conclude the paper with two remarks.
	\begin{itemize}
		\item The proposed DSAV method also applies to the following systems
		\begin{equation*}
			\left\{
			\begin{split}
				\ud Y(t)&=-\nabla V(X(t))\ud t- v Y(t) \ud t+\sigma\ud W(t),\\
				\ud X(t)&= Y(t)\ud t, 
			\end{split}
			\right.
		\end{equation*}	
		where $\sigma\in\mbb R^{d\times d}$ is an invertible constant matrix. Indeed, by denoting $\sigma^{-1} Y(t)=\bar{\mathfrak{P}}(t)$ and $\sigma^{-1}X(t)=\bar{\mathfrak{Q}}(t)$, the above equation can be transformed as the form of \eqref{eq:Langevinorg}. Then correspondingly, one just adapts \eqref{eq:SAV1} and the  DSAV method \eqref{explicitmethod}-\eqref{eq:SAVN} by replacing $\Delta W_n$ with $\sigma \Delta W_n$. 
		
		\item It is straightforward to prove the Feller property of $\{(P_n,Q_n,r_n)\}$, which together Lemma \ref{Numer-expinter} implies that $\{(P_n,Q_n,r_n)\}$ admits at least an invariant measure. Unfortunately, it is hard for us to show the uniqueness of the invariant measure for both $\{(P_n,Q_n,r_n)\}$ and $\{(P(t),Q(t),r(t))\}$, due to the absence of the parabolic H\"ormander's condition for the extended system \eqref{eq:SAV1} (see Remark \ref{rem:exact-PDF}). With the uniqueness of the invariant measure established, the error analysis between the numerical and exact invariant measures will become an intriguing subject of study.
	\end{itemize}

	\setcounter{equation}{0}
	\setcounter{subsection}{0}
	\setcounter{Def}{0}
	\renewcommand{\theDef}{A.\arabic{Def}}
	\renewcommand{\theequation}{A.\arabic{equation}}
	\renewcommand{\thesubsection}{A.\arabic{subsection}}
	\section*{Appendix A. Proof of Lemma \ref{weakorder}}
	\begin{proof}
		\textbf{Step1: We give the weak error decomposition.}	\\
		Substituting \eqref{eq:rn+1-old} into \eqref{eq:Pn+1} and using \eqref{eq:G2}, we have
		\begin{equation}\label{eq:Pn+1new}
			P_{n+1}=
			\big(-2\tau^2G_0(Q_n)^{-1} G_2(Q_n) G_2(Q_n)^\top+ f_{v,\tau}I\big) \Delta W_n+P_n+\tau\mathcal{P}(P_n,Q_n,r_n),
		\end{equation}
		where $\mathcal{P}(p,q,r):=\frac{f_{v,\tau}-1}{\tau}p-2 G_2(q)G_0(q)^{-1}r_n
		-2 G_2(q)G_0(q)^{-1}G_1(q,p)+ f_{v,\tau}(\frac{v^2}{4} -\alpha)(1+\frac{v}{2}\tau)^{-1}q.$
		Moreover, by \eqref{eq:Qn+1} and \eqref{eq:Pn+1new}, we have
		$	Q_{n+1}=Q_n+\tau\mathcal{Q}(P_n,Q_n,r_n)
		+\tau\left(1+\frac{v}{2}\tau\right)^{-1}\left\{-2\tau^2G_0(Q_n)^{-1} G_2(Q_n) G_2(Q_n)^\top+ f_{v,\tau}I\right\} \Delta W_n,$
		where the function \\$\mathcal{Q}(p,q,r):=-\frac{v}{2}\left(1+\frac{v}{2}\tau\right)^{-1}q+\left(1+\frac{v}{2}\tau\right)^{-1}(p+\tau\mathcal{P}(p,q,r))$.
		Letting $Y_n=(P_n^\top,Q_n^\top)^\top$, we have $	Y_{n+1}=Y_n+\tau f_n+g_n \Delta W_n$, where the drift coefficient $f_n=(\mathcal{P}(P_n,Q_n,r_n)^\top,\\\mathcal{Q}(P_n,Q_n,r_n) ^\top)^\top$ and the diffusion coefficient
		\begin{align*}
			g_n:=\left(
			\begin{array}{c}
				I \\ \tau\left(1+\frac{v}{2}\tau\right)^{-1}I
			\end{array}
			\right)\left\{-2\tau^2G_0(Q_n)^{-1} G_2(Q_n)G_2(Q_n)^\top+ f_{v,\tau}I\right\}. 
		\end{align*} 
		Define the stochastic process $Y^\tau(s)=Y_n+\int_{t_n}^s f_n \ud s+\int_{t_n}^{s} g_n \ud W(s)$ for $s\in(t_{n},t_{n+1}]$ and $n\in\mbb N$. Notice that $Y^\tau(t_n)=Y_n$ for all $n\in\mbb N$.

		By the It\^o formula and \eqref{eq:Kol}, we have the following weak error decomposition:
		\begin{align}\label{errordecom}
			&\quad~\E[\psi(P_N,Q_N)]-\E[\psi(\mathfrak P(t_N,P_0,Q_0),\mathfrak Q(t_N,P_0,Q_0))] \notag\\
			&=\E[u^\psi(0,P_N,Q_N)]-\E[u^\psi(t_N,P_0,Q_0)]\\
			&=\sum_{n=0}^{N-1}\E\left[u^\psi(t_N-t_{n+1},Y^\tau(t_{n+1}))-u^\psi(t_N-t_{n},Y^\tau(t_n))\right]\notag\\
			&=\sum_{n=0}^{N-1}\E\int_{t_n}^{t_{n+1}}\Big\{-\partial_t u^\psi(t_N-s,Y^\tau(s))
			+\langle \nabla u^\psi(t_N-s,Y^\tau(s)),f_n\rangle\notag\\
			&\qquad\qquad\qquad\qquad+\frac12\mathrm{tr}(g_ng_n^\top \nabla^2u^\psi(t_N-s,Y^\tau(s)))\Big\}\ud s\notag\\
			&=\sum_{n=0}^{N-1}\E\int_{t_n}^{t_{n+1}}\Big\{\left\langle \nabla u^\psi(t_N-s,Y^\tau(s)),f_n-B(Y^\tau(s))\right\rangle \notag\\
			&\qquad\qquad\qquad\qquad+\frac12\mathrm{tr}\left((g_ng_n^\top-GG^\top)\nabla^2u^\psi(t_N-s,Y^\tau(s))\right)\Big\}\ud s. \notag
		\end{align}
		
		\textbf{Step2: We estimate the right-hand side term of \eqref{errordecom}.}	\\	
		By the It\^o formula, it holds that for any $s\in(t_n,t_{n+1})$, 
		\begin{align}\label{fn-BYtaus}
			f_n-B(Y^\tau(s))=f_n-B(Y_n)+J_1(s)-\int_{t_n}^s \langle \nabla B(Y^\tau(r)), g_n \ud W(r)\rangle,
		\end{align}
		where	$J_1(s):=-\int_{t_n}^s \big(\langle \nabla B(Y^\tau(r)), f_n\rangle+\frac12\mathrm{tr}(g_ng_n^\top \nabla^2B(Y^\tau(r)))\big)\ud r$.

		Denote by $\big(f_n-B(Y_n)\big)_1$  (resp.\ $\big(f_n-B(Y_n)\big)_2$) the first $d$ (resp.\ the last $d$) components of $f_n-B(Y_n)$. 
		Then we have
		$		\big(f_n-B(Y_n)\big)_2=(1+\frac{v}{2}\tau)^{-1}\tau\big(-\frac{v}{2}P_n+\frac{v^2}{4}Q_n+\mcal P(P_n,Q_n,r_n)\big).
		$
		Lemma \ref{cor:Pnunif.} and $\mcal P\in C^\infty_{\rm p}(\mbb R^{2d+1};\mbb R^d)$ yield that for any $p\ge 1$,
		\begin{align}\label{sec4eq8}
			\sup_{n\ge 0}\mbb E\|\big(f_n-B(Y_n)\big)_2\|^p\le K\tau ^p.
		\end{align}
		
		By the definition of $\mcal P$ and $f_{v,\tau}$, we have
		$\big(f_n-B(Y_n)\big)_1=I_1+I_2,$
		where
		\begin{align*}
			I_1:= &\;(1+v\tau+\alpha\tau^2)^{-1}\big(\frac{v^2}{2}-\alpha+\frac{\alpha v}{2}\tau\big)\tau P_n
			-2\tau G_2(Q_n)G_0(Q_n)^{-1}\big(G_1(Q_n,P_n)/\tau)\\
			&\;-(\frac{v^2}{4} -\alpha)(1+v\tau+\alpha\tau^2)^{-1}\left(v+\alpha\tau\right)\tau Q_n,\\
			I_2:= &\left(\nabla V(Q_n)-\alpha Q_n\right)-2 G_2(Q_n)G_0(Q_n)^{-1}r_n.
		\end{align*}	
		It follows from $V\in C^\infty_{\rm p}(\mbb R^d;\mbb R)$ and Corollary \ref{cor:Pnunif.} that for any $p\ge1$,
		\begin{align}\label{eq:I1p}
			\E[\| I_1\|^p]\le K\tau^p.
		\end{align}
		By \eqref{eq:G0} and \eqref{eq:G2}, 
		\begin{align*}
			&I_2=\frac{\nabla V(Q_n)-\alpha Q_n}{\sqrt{U_\alpha(Q_n)}}\left(\sqrt{U_\alpha(Q_n)}-r_n\right)\\
			&\quad+\frac{ r_n(\nabla V(Q_n)-\alpha Q_n)}{\sqrt{U_\alpha(Q_n)}}\Big(1-f_{v,\tau}\big(1+\frac{v}{4}\tau\beta_0+\frac{\tau^2f_{v,\tau}}{2U_\alpha(Q_n)} \|\nabla V(Q_n)-\alpha Q_n\|^2\big)^{-1}\Big).
		\end{align*}
		A direct computation leads to
		\begin{align*}
			&\;1-f_{v,\tau}\Big(1+\frac{v}{4}\tau\beta_0+\frac{\tau^2f_{v,\tau}}{2U_\alpha(Q_n)} \|\nabla V(Q_n)-\alpha Q_n\|^2\Big)^{-1}\\
			=&\;\Big(1+\frac{v}{4}\tau\beta_0+\frac{\tau^2f_{v,\tau}}{2U_\alpha(Q_n)} \|\nabla V(Q_n)-\alpha Q_n\|^2\Big)^{-1}(1+v\tau+\alpha\tau^2)^{-1}\\
			&\;\times\Big(\big(1+\frac{v}{4}\tau\beta_0+\frac{\tau^2f_{v,\tau}}{2U_\alpha(Q_n)} \|\nabla V(Q_n)-\alpha Q_n\|^2\big)(1+v\tau+\alpha\tau^2)-(1+\frac{v}{2}\tau)\Big).
		\end{align*}
		It is observed that the last term in the above formula is of first order in terms of $\tau$.	Hence applying the H\"older inequality, Corollary \ref{cor:Pnunif.}, and Lemma \ref{lem:rn-sqrt}, we obtain that 		$	\E[\|I_2\|^p] \le K\tau^p$ for $1\le p<2$,
		which, in combination with \eqref{eq:I1p}, yields that 
		$
		\E\left[\| \big(f_n-B(Y_n)\big)_1\|^p\right]\le K\tau^p$ for any $1\le p<2.$
		This, along with \eqref{sec4eq8}, gives that for any $1\le p<2$,
		\begin{align}\label{fn-BYn}
			\E\left[\| f_n-B(Y_n)\|^p\right]\le K\tau^p.
		\end{align}
		
		It follows from $V\in C^\infty_{\rm p}(\mbb R^d;\mbb R)$ and Corollary \ref{cor:Pnunif.} that for any $p\ge1$, there exists a constant $K(p)$ such that for any $s\in(t_n,t_{n+1}]$,
		$\E\left[\|J_1(s)\|^p\right]\le K(p)\tau^p$.
		Hence, by the H\"older inequality, Corollary \ref{cor:Pnunif.}, Lemma \ref{lem:exp}, and \eqref{fn-BYn}, 
		\begin{align}\label{eq:fnB}
			&\;\Big|\sum_{n=0}^{N-1}\E\int_{t_n}^{t_{n+1}}\Big\langle \nabla u^\psi(t_N-s,Y^\tau(s)),f_n-B(Y_n)+J_1(s)\Big\rangle\ud s\Big|\\\notag
			\le&\; K \sum_{n=0}^{N-1}\int_{t_n}^{t_{n+1}}\E\left[(1+\|Y^\tau(s)\|^{\eta(1)})e^{-\gamma(t_N-s)}\| f_n-B(Y_n)+J_1(s)\|\right]\ud s\\
			\le&\; K\tau\int_0^{t_N}e^{-\gamma(t_N-s)}\ud s\le K\tau.\notag
		\end{align}			
		In addition, the mean value theorem gives		
		\begin{align*}
			&\;\E\int_{t_n}^{t_{n+1}}\Big\langle \nabla u^\psi(t_N-s,Y^\tau(s)),\int_{t_n}^s \langle \nabla B(Y^\tau(r)), g_n \ud W(r)\rangle\Big\rangle\ud s\\
			=&\;\E\Big[\E \Big(\int_{t_n}^{t_{n+1}}\Big\langle \nabla u^\psi(t_N-s,Y^\tau(t_n)),\int_{t_n}^s \langle \nabla B(Y^\tau(r)), g_n \ud W(r)\rangle\Big\rangle\ud s\Big| \mcal F_{t_n}\Big)\Big]\\
			&\;+\E\int_{t_n}^{t_{n+1}} \int_0^1D^2u^\psi(t_N-s,\theta_1 Y^\tau(s)+(1-\theta_1)Y^\tau(t_n))\\
			&\qquad\qquad\qquad\Big( Y^\tau(s)-Y^\tau(t_n),\int_{t_n}^s \langle \nabla B(Y^\tau(r)), g_n \ud W(r)\rangle\Big)\ud \theta_1\ud s\\
			=&\;\E\int_{t_n}^{t_{n+1}} \int_0^1D^2u^\psi(t_N-s,\theta_1 Y^\tau(s)+(1-\theta_1)Y^\tau(t_n))\\
			&\qquad\qquad\qquad\Big( Y^\tau(s)-Y^\tau(t_n),\int_{t_n}^s \langle \nabla B(Y^\tau(r)), g_n \ud W(r)\rangle\Big)\ud \theta_1\ud s.
		\end{align*}
		By $V\in C^\infty_{\rm p}(\mbb R^d;\mbb R)$ and Corollary \ref{cor:Pnunif.}, we infer that for any $p\ge1$ and $s\in(t_n,t_{n+1}]$, 
		\begin{align*}
			\E\left[\Big\|\int_{t_n}^s\langle \nabla B(Y^\tau(r)), g_n \ud W(r)\rangle\Big\|^p\right]+\E\left[\|Y^\tau(s)-Y^\tau(t_n)\|^p\right]\le K\tau^{\frac{p}{2}}.
		\end{align*}
		Consequently, by H\"older's inequality and Lemma \ref{lem:exp},
		\begin{align*}
			&\left|\sum_{n=0}^{N-1}\E\int_{t_n}^{t_{n+1}}\left\langle \nabla u^\psi(t_N-s,Y^\tau(s)),\int_{t_n}^s \langle \nabla B(Y^\tau(r)), g_n \ud W(r)\rangle\right\rangle\ud s\right|\\
			&\le K\tau\sum_{n=0}^{N-1}\int_{t_n}^{t_{n+1}} \!\!\!\!\int_0^1\!\!\left(\E\big(1\!+\|\theta_1 Y^\tau(s)\!+\!(1-\theta_1)Y^\tau(t_n)\|^{\eta(2)}\big)^2\right)^{\frac12}e^{-\gamma(t_N-s)}\ud \theta_1\ud s
			\le K\tau.
		\end{align*}
		This inequality, together with \eqref{fn-BYtaus} and \eqref{eq:fnB}, yields
		\begin{align}\label{estimate1}
			\left|\sum_{n=0}^{N-1}\E\int_{t_n}^{t_{n+1}}\left\langle \nabla u^\psi(t_N-s,Y^\tau(s)),f_n-B(Y^\tau(s))\right\rangle
			\ud s\right|\le K\tau.
		\end{align}
		It remains to estimate
		$\sum_{n=0}^{N-1}\E\int_{t_n}^{t_{n+1}}\frac12\mathrm{tr}\big((g_ng_n^\top-GG^\top)\nabla^2u^\psi(t_N-s,Y^\tau(s))\big)\ud s.$
		Notice that 
		\begin{align*}
			g_n-G=\left(
			\begin{array}{c}
				-2 \tau^2G_0(Q_n)^{-1} G_2(Q_n) G_2(Q_n)^\top+ (f_{v,\tau}-1)I \\ \tau\left(1+\frac{v}{2}\tau\right)^{-1}
				\left(-2 \tau^2G_0(Q_n)^{-1} G_2(Q_n) G_2(Q_n)^\top+ f_{v,\tau}I\right)
			\end{array}
			\right).
		\end{align*}
		It is easy to verify that for any $p\ge1$, there exists a constant $K(p)>0$ such that 
		$$\E\left[\|g_n-G\|^p\right]\le K(p)\tau^p,\quad \E\left[\|g_n\|^p\right]\le K(p),\quad n\in \N.$$
		Applying H\"older's inequality and Lemma \ref{lem:exp}, we have
		\begin{align*} 
			\left|\sum_{n=0}^{N-1}\E\int_{t_n}^{t_{n+1}}\frac12\mathrm{tr}\left((g_ng_n^\top-GG^\top)\nabla^2u^\psi(t_N-s,Y^\tau(s))\right)\ud s\right|\le K(p)\tau,
		\end{align*}
		which, combined with \eqref{errordecom} and \eqref{estimate1}, yields the desired result. 
	\end{proof}


\bibliographystyle{abbrv}
\bibliography{mybibfile}

\end{document}